\documentclass[reqno,11pt]{amsart}
\usepackage{amsmath, amssymb, amsthm,amsfonts} 
\usepackage[english]{babel}
\usepackage{bbm}
\usepackage{graphicx}
\usepackage{url}
\usepackage{epstopdf}
\usepackage[a4paper,bindingoffset=0.5cm,left=2cm,right=2cm,top=2.5cm,bottom=2cm,footskip=.8cm]{geometry}

\usepackage{tikz}
\usetikzlibrary{arrows, automata,positioning,calc,shapes,decorations.pathreplacing,decorations.markings,shapes.misc,petri,topaths}
  \usetikzlibrary{plotmarks}
  \usepackage{grffile}
\newlength\figureheight
  \newlength\figurewidth
\usetikzlibrary{math,shapes.geometric}

\usepackage{rotating}
\usepackage{amsbsy,enumerate}
\usepackage{graphicx}
\usepackage{ccaption}
\usepackage{comment}
\usepackage{mathrsfs} 
\usepackage{diagbox}

\newcommand{\s}{^\star}

\newcommand{\bs}{\boldsymbol}

\newcommand{\vb}{\vspace{3.2mm}}

\allowdisplaybreaks

\newtheorem{lemma}{Lemma}

\newtheorem{theorem}{Theorem}

\newtheorem{example}{Example}

\newtheorem{proposition}{Proposition}

\makeatletter
\renewcommand{\fnum@figure}[1]{\textbf{\figurename~\thefigure}. }
\renewcommand{\fnum@table}[1]{\textbf{\tablename~\thetable}. }
\makeatother

\begin{document}

\title[Cram\'er-Lundberg model
with fluctuating number of clients]{The Cram\'er-Lundberg model\\
with a fluctuating number of clients}

\author{Peter Braunsteins \& Michel Mandjes}

\begin{abstract}
This paper considers the Cram\'er-Lundberg model, with the additional feature that the number of clients {can} fluctuate over time. 
Clients arrive according to a Poisson process, where the times they spend in the system form a sequence of independent and identically distributed non-negative random variables. While in the system, every client generates claims and pays premiums.
In order to describe the model's rare-event behaviour, we establish a sample-path large-{deviation} principle. 
{This describes the joint rare-event behaviour of the  reserve-level process and the client-population size process. The large-deviation principle  can be used to determine the decay rate of the time-dependent ruin probability as well as the most likely path to ruin.}
{Our results allow us to determine whether the chance of ruin is greater with more or with fewer clients and, more generally, to determine to what extent a large deviation in the reserve-level process can be attributed to an unusual outcome of the client-population size process.}
%The sample-path large-deviation principle  in addition provides us with the most likely path to ruin, thus shedding light on the question to what extent ruin can be attributed to the client-level dynamics, and to what extent to the an unusual realization of the claim process as generated by the clients. 

\vb


\vb

\noindent
{\sc Affiliations.} The authors  are with the Korteweg-de Vries Institute for Mathematics, University of Amsterdam, Science Park 904, 1098 XH Amsterdam, The Netherlands. PB is also with the Mathematical Institute, P.O. Box 9512,
2300 RA Leiden,
The Netherlands.
MM is also with E{\sc urandom}, Eindhoven University of Technology, Eindhoven, The Netherlands, and Amsterdam Business School, Faculty of Economics and Business, University of Amsterdam, Amsterdam, The Netherlands. Both authors'  research has been funded by the NWO Gravitation project N{\sc etworks}, grant number 024.002.003. 
Email: \url{m.r.h.mandjes@uva.nl}, \url{pbraunsteins@gmail.com}.

\vb

\noindent
{\it Date}: {\today}.

\end{abstract}

\maketitle

\section{Introduction}

The {\it Cram\'er-Lundberg} (CL) model \cite{C1, L1, L2} plays a pivotal role in ruin theory. {It is} a stochastic process that represents the evolution of {an} insurance firm's reserve  level (also referred to as surplus-level process). The
{primary goal} is to evaluate the {{\it ruin probability} for a given initial surplus $u$}, i.e., the probability that the reserve-level process drops below 0. 
In the most basic variant of the CL model, claims are independent and identically distributed (iid)~non-negative random quantities that arrive according to a Poisson process (with rate {$\nu>0$}), while premiums are earned at a deterministic linear rate $r>0$.
For this base model a broad range of results have been obtained, most notably a characterisation of the ruin probability through its Laplace transform. 
In addition, relying on elements from {\it large-deviations theory}, the asymptotics of the ruin probability were identified \textcolor{black}{for large values of the initial surplus $u$}, an important observation in this context being that with overwhelming probability the path to ruin is by approximation linear (under the proviso that the claims are light-tailed). 
For more background on these results, and an account of the area of ruin theory in general, we refer to e.g.\ \cite{AA, MIK, KGS, TEU}.

The CL model described above is admittedly a gross simplification of reality, in that various features that play a role in practice are not incorporated. 
This realization led to a stream of results that in various directions generalize the classical setup. 
Without attempting to provide an exhaustive overview, we {now} briefly mention a few of the main strands of research. 
Arguably the most important extension concerns the time-dependent ruin probability, i.e., the probability that the reserve level becomes negative before a given point in time. 
We refer to \cite[Ch. V]{AA} for an overview of results in this area; notably, under the large-deviations scaling (with light-tailed claims) the most likely path to ruin is still {linear}. 
{Steps have also been taken to generalise the claim arrival process, which is traditionally of compound Poisson type. 
In \cite{G1} a diffusion term is added, and (more generally) in \cite{DebM, KYP} the reserve level evolves as a L\'evy process.}
In e.g.\ \cite{TOR} the arrival process is assumed to be of Hawkes type.
% \textcolor{black}{While} in the \textcolor{black}{standard} CL framework the claim arrival process is of compound Poisson type, extensions have been considered in which a diffusion term was added \cite{G1}, or (more generally) in which the reserve level evolves as a L\'evy process \cite{DebM, KYP}. 
Other extensions {include} variants in which the insurance firm's interest income is incorporated; see for instance \cite{AC, BM}, and the textbook treatment in \cite[Ch.\ VIII]{AA}.
We finally mention the branch of the literature in which the reserve process is modulated by a background process; see e.g.\ the Markov-modulated framework in \cite[Ch. VII]{AA} and the mixing model in \cite{CDMR}.

In the present paper we consider another extension of the CL model, namely a model in which the insurance firm has a stochastically fluctuating number of clients. 
One could view the standard CL model as a setup in which the number of clients is {fixed}, while in practice, so as to properly assess the ruin probability, one should evidently take into account variations in the client population size. 
We model the client-level fluctuations by letting clients arrive according to a Poisson process, where the times they spend in the system (as a client of the insurance firm, that is) form a sequence of iid non-negative random variables; while in the system, each client generates iid claims at Poisson instants, and pays premiums at a rate $r$.

In the {CL model with a fluctuating number of clients}, we wish to assess the {time-dependent} ruin probability, given the insurance firm's initial surplus. We do so in an asymptotic context, corresponding to the (realistic) situation that the insurance firm's client base is consistently large. Concretely, 
we let the (Poissonian) client arrival rate be $n\lambda$ {and \textcolor{black}{the initial surplus be $n u$ for some $u>0$}}, where $n$ is a scaling parameter that we let grow large. In this limiting setting we {derive a sample-path large-deviation principle (LDP).
This sample-path LDP is bivariate, in that it jointly describes the reserve-level process and the client-population-size process. 
We use it not only to evaluate the logarithmic decay rate of the time-dependent ruin probability, but also to investigate two questions about the most likely path to ruin: (i) is the chance of ruin greater when the client population is higher or lower than expected?; (ii) to what extent can a large deviation in the reserve level process be attributed to an unusual outcome of the client population process?
}
% are able to evaluate the logarithmic decay rate (in $n$, that is) of the probability of ruin before time $T>0$, given that the initial surplus in $nu$ (for an arbitrarily given value of $u>0$).
% A relevant question from a large-deviations perspective is: when studying a trajectory towards ruin, is ruin essentially caused by (i)~a lower number of clients than usual, thus generating a lower volume of income due to premiums, (ii)~a higher number of clients than usual, thus generating more claims than usual, {(iii)~a roughly normal number of clients, but each of them generating a higher claimed amount (potentially due to filing unusually many claims, or by filing claims that are systematically larger than usual)?}
% A relevant question from a large-deviations perspective is: when studying a trajectory towards ruin, is ruin essentially caused by (i)~a lower number of clients than usual, thus generating a lower volume of income due to premiums, (ii)~a higher number of clients than usual, thus generating more claims than usual, (iii)~a roughly normal number of clients, but each of them generating a higher claimed amount (potentially due to filing unusually many claims, or by filing claims that are systematically larger than usual), or a combination of these factors?

At {a} technical level, {the} crucial difference with the conventional CL model, where the number of clients is fixed, is that when we allow the number of clients to fluctuate, the increments of the reserve level process are no longer independent. 
Traditionally, sample-path large deviations mainly focus on settings with independent increments. Results in this area essentially go back to an early paper by Varadhan \cite{VAR}; see also the contributions {in} \cite{BOR65, MOG1, MOG2}. 
{Indeed, the sample path LDP for the standard CL model is implied by the classical result of Mogulskii \cite[Thm.\ 5.1.2]{DZ}.} Models in which there is correlation in the increment process are substantially harder to deal with, but often offer richer behaviour. 
For example, in the CL model with a fluctuating number of clients the most likely path to ruin is no longer linear.  
For work on sample-path large deviations for processes with dependent increments we refer to (the generalised version of) Schilder's theorem for Gaussian processes, which was established in \cite{AZEN, BZ}; see also the textbook treatment in \cite{DS}. 
This type of result has been applied extensively in the operations research domain, addressing various rare-event related problems concerning Gaussian storage systems \cite{ADM,MMNU, MNG, MU, MN}. 
{We also point to} sample-path LDPs for specific queueing models which can be found in e.g.\ \cite{BCL,SW,WIS}.

% We prove the sample-path LDP for our client-level CL model following the procedure outlined in \cite{FK}. %%%%\cite{FK} outline a process version of the weak convergence approach.
We prove the LDP for our variant of the CL model by first establishing an LDP that corresponds to a single point in time, then extending this to an LDP for multiple points in time, before finally establishing the full sample-path LDP.
In this approach, the first two steps {rely on a fundamental observation: even though its increments are not independent, it is possible to decompose the process into independent components, thus allowing arguments based on sums of independent random variables to be applied.}
% In this approach, first a single-dimensional LDP is established, which is then extended to a finite-dimensional LDP; these steps are  usually conceptually relatively straightfoward. 
The main technical hurdle lies in the final step: upgrading the finite-dimensional LDP  to a sample-path LDP. This amounts to verifying one of the equivalent exponentially tightness characterisations as provided by \cite[Thm.\ 4.1]{FK}. We point out that since the number of clients fluctuates autonomously (i.e., it is not affected by the reserve-level process), the structure of the LDP resembles the decompositions found in \cite{GANG,LIP}. 

This paper is organised as follows. In Section \ref{Sec:model} we provide a detailed model description of our CL model with a fluctuating client population. 
In Section \ref{Sec:mainr} we present our main results. 
These cover finite-dimensional LDPs as well as the full sample-path LDP. 
In addition, we present results that shed light on the most likely path to ruin, including experimental insight into the most likely cause of ruin. 
Proofs are provided in Section \ref{Sec:proofs}: first we focus on establishing finite-dimensional LDPs, {and then} extend these to the full sample-path LDP by relying on a tightness argument.

\section{Model}\label{Sec:model}
In this section we introduce the CL model with a fluctuating client population. In this model description, we distinguish between the dynamics of the population size, and the dynamics corresponding to each individual client in the system.

\vb

\noindent {\it Client-population-size dynamics.}
Clients arrive to the system according to a Poisson process with rate $n\lambda$. Here $\lambda$ is a positive parameter, and $n$ is a scaling parameter that we let grow large. 
The clients stay in the system for independent and identically distributed (iid) amounts of time, in the sequel referred to as the clients' {\it sojourn times}. {For convenience, in our analysis we let the sojourn times have density $h(\cdot)$, but our arguments hold more generally (in particular allowing for both continuous and discrete sojourn-times distributions).}
\textcolor{black}{In queueing-theoretic terminology, the number of clients simultaneously present follows the dynamics of a so-called M/G/$\infty$ system.}
\iffalse\footnote{\small PB: I don't think we want to assume a density here, as the arguments below hold more generally. Perhaps this can be replaced by `have a distribution function $H(\cdot)$. Note that this means we allow for both continuous and discrete distributions for the sojourn time.'? 
MM: I adapted it -- do you agree?} \fi

\noindent
At time $0$, the number of clients already present is $n f_0$ for some $f_0\geqslant 0$.  
These $nf_0$ clients have remaining sojourn times that are iid with density $h^\circ(\cdot)$. In this respect a {natural choice is to let} the remaining sojourn times have the well-known {\it excess lifetime distribution}, i.e., for $t\geqslant 0$,
\[h^\circ(t) = \int_t^\infty h(s)\,{\rm d}s\left/\int_0^\infty sh(s)\,{\rm d}s\right.,\]
where in the denominator we recognize the mean duration of a `fresh' sojourn time; it is easily verified that this density integrates to $1$.

\textcolor{black}{Recall that for the M/G/$\infty$ system in equilibrium, the number of clients simultaneously present has a Poisson distribution with mean
\[ \lambda \int_0^\infty sh(s)\,{\rm d}s.\]
Moreover, their remaining service times are independent and obey the excess lifetime distribution, independently of the number of clients present.}

 Throughout we impose the mild technical assumption that remaining sojourn times have a uniformly bounded density, i.e., that there exists a constant $C<\infty$ such that  $h^\circ(t) \leqslant C$ for all $t \geqslant 0$. Note that this assumption holds if $h^\circ(\cdot)$ is the excess lifetime distribution (as then we have that $h^\circ(t)$ is, for any $t\geqslant 0$, majorised by the  multiplicative inverse of the mean of a `fresh' sojourn time).
 A minor technical remark is that, for convenience, the number $n f_0$ is throughout assumed to be an integer, but in the case it is not integer our analysis can be adapted easily by a straightforward rounding procedure.

\noindent
\textcolor{black}{Let $(T_{i,n})_{i \geqslant 1}$ denote the sequence of iid exponentially distributed random variables describing the clients' arrival times in the $n$-th process and $(\mathcal{N}_{t,n})_{t\geqslant 0}$ the corresponding renewal process (i.e., a Poisson process of rate $n\lambda$).}
\textcolor{black}{Let $(\tau_i)_{i \geqslant 1}$ denote the iid sequence of sojourn times and let $(\tau^\circ_i)_{i \geqslant 1}$ denote the iid sequence of remaining sojourn times. The number of clients present at $t$ is then
\begin{align}\notag
F_n(t) &:= nf_0 - \sum_{i=1}^{nf_0} 1\{\tau^\circ_i \leqslant t \} + \mathcal{N}_{t,n} - \sum_{i=1}^{\mathcal{N}_{t,n}} 1\{T_{i,n} + \tau_{i} \leqslant t\} \\
&= \sum^{nf_0}_{i=1} 1\{\tau^\circ_i >t \} + \sum_{i=1}^{\mathcal{N}_{t,n}} 1\{T_{i,n} + \tau_{i} > t\}.\label{Fndef}
\end{align}}
% In the sequel, we let $F_n(t)$ denote the number of clients present at time~$t\geqslant 0$. 

Notice that $F_n(t)$ consists of both clients who belonged to the initial $nf_0$ clients (and have not left yet by time $t$) and clients who 
arrived in $(0,t]$ (and are still present at time $t$). 
We denote the corresponding normalised process by
\[\bar F_n(t) := \frac{F_n(t)}{n}.\]

\noindent  {\it Client behaviour.}
Now that we have introduced the stochastic mechanism that generates the client-population dynamics, we continue by focusing on the behaviour of each individual client while being in the system. 
During her sojourn time a client pays a constant premium rate of $r>0$ per unit of time. 
Every client generates claims at a Poisson rate $\nu>0$ while in the system. 
The claim sizes form an iid sequence, with the moment generating function (mgf) of an individual claim being denoted by $\beta(\cdot)$.
Throughout we assume that we are in the light-tailed \textcolor{black}{setting}, in that $\beta(\theta)$ is finite for $\theta$ in an open neighborhood of the origin. 
%\footnote{\small MM: Have we sorted this out? PB: Yes. At the end of the tightness section we apply Cramer's theorem, which holds under this assumption.}
The {\it net aggregate claim process} represents the total claimed amount (by the entire population, that is) decreased by the premiums received by the insurance firm. 
\textcolor{black}{Let $({\mathcal M}^{\circ}_{t,i})_{t \geqslant 0}$ and $({\mathcal M}_{t,i})_{t \geqslant 0}$ denote independent sequences of Poisson processes of rate $\nu$, describing the number of claims corresponding to the initially present and arriving customers respectively, and let $(Z^\circ_{k,i})$ and $(Z_{k,i})$ denote sequences of iid random variables describing the $k$-th claim by the initially present and arriving customers respectively. The net aggregate claim process at time $t\geqslant 0$ (with $G_n(0)=0$) is then
\begin{equation}
\begin{aligned}
G_n(t) :=&\: \sum_{i=1}^{nf_0} \bigg( -r(t \wedge \tau^\circ_i) +\sum_{k=1}^{{\mathcal M}^{\circ}_{t \wedge \tau^\circ_i,i}} Z^\circ_{k,i} \bigg)\: \\&\: +\sum_{i=1}^{\mathcal{N}_{t,n} } \bigg( -r( 0 \vee [(t- T_{i,n}) \wedge \tau_i])+ \sum_{k=1}^{{\mathcal M}_{(t-T_{i,n}) \wedge \tau_i,i}} Z_{k,i} \bigg).\label{Gndef}
\end{aligned}
\end{equation}
We denote the corresponding normalised process by}
\[
\bar G_n(t) :=\frac{G_n(t)}{n}.
\]
% We denote by $G_n(t)$ the value of the net aggregate claim process at time $t\geqslant 0$ (with $G_n(0)=0$), and \[\bar G_n(t) :=\frac{G_n(t)}{n}\] is the corresponding normalised process.
%\footnote{\small MM: I think throughout the document we should avoid the word `queue', and where you wrote `service time' we should use `sojourn time'. In the insurance world it's confusing to use the queueing notions `queue' and `service'. I also prefer `client' over `customer'.}
Our goal is to produce a probabilistic description of the object $(F_n(\cdot), G_n(\cdot))$ that allows us to identify the logarithmic decay rate (\textcolor{black}{as $n$ grows large}) of the time-dependent ruin probability
\begin{equation}\label{pnut}p_n(u,T) := {\mathbb P}(\exists t\in[0,T]: \bar G_n(t) \geqslant u),\end{equation}
given that $\bar F_n(0)=f_0$ and $\bar G_n(0) = 0.$

\section{Main results}\label{Sec:mainr}

\subsection{Large-deviation principles} 
Our main result is the sample-path LDP of the bivarate process $(F_n(\cdot), G_n(\cdot))$, to be presented in Theorem \ref{thm:main}. We establish this LDP by first proving more basic, finite-dimensional LDPs, which we then upgrade to the full sample-path LDP through a tightness argument. 
Concretely, we first discuss a one-point LDP (pertaining to a single point in time, that is), then extend this to a finite-point LDP (pertaining to finitely many time epochs), and then finally to a sample-path LDP. In this section we state these results, and provide the main ideas behind the proofs (which are given in detail in 
Section \ref{Sec:proofs}).

It is noted that the process $(F_n(\cdot), G_n(\cdot))$ is not necessarily Markovian --- or, more precisely: only when the clients' sojourn times are exponentially distributed, $(F_n(\cdot), G_n(\cdot))$ is a Markov chain. Importantly, however,  we can still use arguments that are based on sums of independent random variables. 
Two crucial observations in this context are:
\begin{itemize}
\item[$\circ$]
The process $(F_n(\cdot), G_n(\cdot))$, \textcolor{black}{as defined via \eqref{Fndef} and \eqref{Gndef}}, can be decomposed as the sum of two {\it independent} components: one related to the contribution of the $nf_0$ clients who were already present at time $0$, which we denote by $(F^-_n(\cdot), G_n^-(\cdot))$, and one related to the \textcolor{black}{${\mathcal N}_{t,n}$} clients who enter in the interval $(0,t]$, which we denote by $(F_n^+(\cdot),G_n^+(\cdot))$. 
{To be precise, $F^-_n(t)$ is the number of clients who were present at time $0$ who are still present at time $t$, and $G_n^-(t)$ is the net aggregate claim volume up to time $t$ which was generated by the clients who were present at time $0$.
The process $(F_n^+(\cdot),G_n^+(\cdot))$ is defined similarly, but now corresponding to clients who were {\it not} present at time $0$. As a consequence, $F_n^+(t)=F_n(t)-F^-_n(t)$ and $G_n^+(t)=G_n(t)-G^-_n(t)$.}
% In the sequel we refer to these processes as $(F^-_n(\cdot), G_n^-(\cdot))$ and $(F_n^+(\cdot),G_n^+(\cdot))$, respectively.
% \item[$\circ$]
% The process $(F_n(\cdot), G_n(\cdot))$ can be decomposed as the sum of two {\it independent} components: one related to the clients who were already present at time $0$, and one related to the clients that enter in the interval $(0,t]$. 
% In the sequel we refer to these processes as $(F^-_n(\cdot), G_n^-(\cdot))$ and $(F_n^+(\cdot),G_n^+(\cdot))$, respectively.
\item[$\circ$] By a direct application of known thinning and superposition properties of Poisson processes, both $(F^-_n(\cdot), G_n^-(\cdot))$ and $(F^+_n(\cdot), G_n^+(\cdot))$ can be interpreted as sums of iid processes, each of them dirstributed as some $(F^-(\cdot),G^-(\cdot))$ and $(F^+(\cdot),G^+(\cdot))$, respectively. 
More precisely,  $(F_n^-(\cdot),G_n^-(\cdot))$ can be represented by the sum of $n f_0$ iid copies of $(F^-(\cdot),G^-(\cdot))$, where each copy corresponds to the contribution of a single client who is present at time $0$. 
Similarly, $(F^+_n(\cdot),G_n^+(\cdot))$ can be seen as the sum of $n$ iid copies of $(F^+(\cdot), G^+(\cdot))$, where each copy corresponds to the contribution of a stream of clients that arrive according to a Poisson process with rate $\lambda$.
%\footnote{\textcolor{black}{There is a technical issue with splitting the initial individuals into $nf_0$ independent copies, and the arrival processes into $n$ copies -- to apply Cram\'er theorem, as we do below, we need the number of copies of each to be the same, i.e., either split both into $nf_0$ copies or split both into $n$ copies. Below it seems that when we apply Cram\'er's theorem we are splitting them both into $n$ copies. If we do that then we are implicitly assuming that $f_0$ is an integer because if $f_0$ is an integer then we can split the contributions of the initial clients evenly into $n$ groups, otherwise we cannot. The obvious solution is to split Poisson the client arrival stream into $n f_0$ independent Poisson processes with rate $\frac{\lambda}{f_0}$ and then apply Cram\'er's theorem for the sum of $n f_0$ random variables to obtain the same result. Note that we are now splitting the clients into groups of 1 and only require that $n f_0$ is an integer. I'm not sure this is worth mentioning.}}
%(In this reasoning we let $nf_0$ be integer-valued; the construction carries over to the case of $nf_0$ being non-integer by an elementary truncation argument.)
%\footnote{\small \tiny PB: It looks like we are implicitly assuming that $f_0$ is an integer here. I guess that it only requires minor modification to treat the case where $f_0 \in \mathbb{R}_+$ and there are initially $\lfloor f_0 n \rfloor$ customers in the queue.}
\end{itemize}

\subsubsection{One-point LDP} 
We start by deriving a large-deviation principle  for the random vector $(\bar F_n(t), \bar G_n(t))$ for a given time point $t>0$. 
As an immediate consequence of the observations above, we can represent $(\bar F_n(t), \bar G_n(t))$ as a sum of \textcolor{black}{$n$} iid random vectors. We can thus apply Cram\'er's theorem \cite[Section 2.2.2]{DZ},
%\footnote{\textcolor{black}{I guess we are using the multivariate version of Cram\'er's theorem so we need the one in \cite[Section 2.2.2]{DZ}?}}, 
so as to obtain a large-deviation principle  whose rate function is given by the Legendre transform, \textcolor{black}{for $(f,g)\in{\mathbb R}_+\times {\mathbb R}$,}
% \footnote{MM: I have tried to consistently use the order $f,g$ and $\omega,\theta$, as that's in the end the more natural order. I have done this quite precisely, but I may have forgotten one or more. Could you check? In addition, I have tried to stick to UK-English, but usually write US-English, so bear with me. This you may need to check as well -- but don't do a global change, as this will ruin all the itemize's.}
\begin{equation}\label{OP:1}
I_{t}(f,g) = \sup_{\omega, \theta\in{\mathbb R}} \left\{ \omega f + \theta g - f_0 \log M_{t}^-(\omega,\theta) - \log M_t^+(\omega, \theta) \right\},
\end{equation}
where $M_t^i(\omega, \theta) = \mathbb{E} \exp(\omega F^i(t)+\theta G^i(t))$ for $i \in \{-,+\}$. \textcolor{black}{Cram\'er's theorem concretely entails that for a set $B\subset {\mathbb R}_+\times {\mathbb R}$ we have that
\begin{align}
    -\inf_{(f,g)\in B^\circ}I_t(f,g)&\leqslant \liminf_{n\to\infty} \frac{1}{n}\log {\mathbb P}((\bar F_n(t),\bar G_n(t))\in B)\notag\\&\leqslant\limsup_{n\to\infty} \frac{1}{n}\log{\mathbb P}((\bar F_n(t),\bar G_n(t))\in B)\leqslant -\inf_{(f,g)\in \bar B}I_t(f,g),\label{LDP1}
\end{align}
where $B^\circ$ is the interior of the set $B$ and $\bar B$ its closure. It provides an informal justification for the frequently used approximation
\[{\mathbb P}((\bar F_n(t),\bar G_n(t))\in B) \approx \exp\left(-n\inf_{(f,g)\in  B}I_t(f,g)\right).\]} 

The next step is to compute the mgf\,s $M^+_t(\omega, \theta)$ and $M^-_t(\omega, \theta)$. 
To this end, observe that the net claim process of an individual client (while in the system) is a L\'evy process \cite{DebM}, viz.\ a compound Poisson process with drift, say $Z(\cdot)$. It is directly verified that  the mgf of $Z(t)$ can be written as $(\varphi(\theta))^t$, where
\begin{equation}\label{OP:2}
\varphi(\theta) = \mathbb{E} \exp ( \theta Z(1)) = \exp (-r \theta + \nu( \beta (\theta ) - 1)).
\end{equation}
To compute $M^-_t(\omega, \theta)$, let $\tau^\circ$ be the random variable corresponding to a typical residual sojourn-time duration of  a client who is present at time $0$.
Conditioning on the time this client leaves, we readily obtain
\begin{equation}\label{OP:3}
M_t^-(\omega, \theta) ={\mathbb E}\, e^{\theta Z(\tau^\circ\wedge t)}e^{\omega 1\{\tau^\circ >t\}}= \int_0^t h^\circ(s) \, (\varphi(\theta))^s \, {\rm d}s+ (\varphi(\theta))^t e^\omega \int_t^\infty h^\circ(s) \, {\rm d}s.
\end{equation}
The next goal is to compute $M^+_t(\omega, \theta)$. To this end, we rely on the property that the number of clients that arrive in the interval $(0,t]$ is Poisson with parameter $\lambda t$. In addition, conditional on the number of arrivals, \textcolor{black}{the arrival times can be seen as order statistics of a sequence of iid uniformly distributed random variables, see for example \cite[p.\ 303]{Ross}.}
% , the arrival epochs are independent and uniformly distributed on $(0,t]$. 
We thus find, with $U$ being uniformly distributed on $[0,1]$ and $\tau$ the random variable corresponding to a typical duration of the time a client spends in the system,
\begin{align*}M^+_t(\omega, \theta) &
= \sum_{k=0}^\infty e^{-\lambda t}\frac{(\lambda t)^k}{k!}\left({\mathbb E}\, e^{\theta Z(\tau\wedge t(1-U))}e^{\omega 1\{\tau > t(1-U)\}} \right)^k
\\
&=\sum_{k=0}^\infty e^{-\lambda t}\frac{(\lambda t)^k}{k!}\left(
\int_0^t \frac{1}{t}\left(
\int_0^{t-s}h(r)\, (\varphi(\theta))^{r} \,{\rm d}r + (\varphi(\theta))^{t-s} e^\omega \int_{t-s}^\infty h(r)\,{\rm d}r 
\right){\rm d}s
\right)^k,\end{align*}
which simplifies to
\begin{equation}\label{OP:4}
\exp\left(\lambda\left(\int_0^t \left(
\int_0^{t-s} h(r)\,(\varphi(\theta))^{r} \,{\rm d}r + (\varphi(\theta))^{t-s} e^\omega \int_{t-s}^\infty h(r)\,{\rm d}r 
\right){\rm d}s-1\right)\right).
\end{equation}
Upon combining the above, we have thus established the following result.%\footnote{\small MM: add details on topology and stuff? Or is that too evident here? (I actually would say so.)}
\begin{proposition}\label{P1}
The pair $(\bar F_n(t),\bar G_n(t))$ satisfies the LDP with rate $n$ and rate function $I_t(f,g)$ characterised by \eqref{OP:1}--\eqref{OP:4}.
\end{proposition}

\subsubsection{Multi-point LDP} 
We proceed by deriving a multi-point LDP, i.e., an LDP for the $2d$-dimensional random vector \[(\bar F_n(t_1), \dots , \bar F_n(t_d), \bar G_n(t_1), \dots , \bar G_n(t_d) )\] where $0\leqslant t_1< t_2< \ldots < t_d$ and $d\in{\mathbb N}$.
\textcolor{black}{This can be seen as the $d$-dimensional counterpart of the LDP above: the 2-dimensional vector $(\bar F_n(t),\bar G_n(t))$ has to be replaced by the $2d$-dimensional vector  $(\bar F_n(t_1), \dots , \bar F_n(t_d), \bar G_n(t_1), \dots , \bar G_n(t_d) )$ in \eqref{LDP1}.}

Mimicking the argumentation used in the case $d=1$, we now apply the $2d$-variate version of Cram\'er's theorem \cite[Section 2.2.2]{DZ} to obtain an LDP with rate function, \textcolor{black}{for $({\bs f},{\bs g})\in {\mathbb R}^d_+\times {\mathbb R}^d$,}
\[I_{\bs t}({\bs f},{\bs g})=\sup_{\bs \theta,\bs \omega\in{\mathbb R}^d} \left(\sum_{j=1}^d \omega_j f_j+\sum_{j=1}^d \theta_j g_j - f_0\,\log M^-_{\bs t}(\bs\omega, \bs\theta) - \log M^+_{\bs t}(\bs\omega, \bs\theta)\right),\]
where 
\begin{equation}\label{eq:Mdef}
M^i_{\bs t}(\bs\omega, \bs\theta) =  {\mathbb E}\exp\left(
\sum_{j=1}^d \omega_j F^i(t_j)+\sum_{j=1}^d \theta_j G^i(t_j)
\right),
\end{equation} for $i\in\{-,+\}$.  
Note that, as before, we split the required mgf into one representing the contribution of the clients present at time $0$ and another corresponding to the contribution of the clients arriving in $(0,t]$, with these two contributions being independent.
We can derive the mgf\,s $M^-_{\bs t}(\bs\omega, \bs\theta)$ and $M^+_{\bs t}(\bs\omega, \bs\theta)$ by following a similar method to the one used in the one-point case; however, due to the non-Markovian nature of the process, this derivation is relatively involved and is therefore postponed to Section \ref{Sec:FDLDP}. 
We thus establish the following result.

\begin{proposition}\label{P2}
The vector  $(\bar F_n(t_1), \dots , \bar F_n(t_d), \bar G_n(t_1), \dots , \bar G_n(t_d) )$ satisfies the LDP with rate $n$ and rate function $I_{\bs t}({\bs f},{\bs g})$, where $M^-_{\bs t}(\bs\omega, \bs\theta)$ and $M^+_{\bs t}(\bs\omega, \bs\theta)$ are given in Lemmas $\ref{L1}$ and $\ref{L2}$, respectively.
%$M^-_{\bs t}(\bs\omega, \bs\theta)$ is given in \eqref{Mme} and $M^+_{\bs t}(\bs\omega, \bs\theta)$ can be recovered from \eqref{Mp}--\eqref{Mbpl3}.
\end{proposition}

%In Section ... we state our main result, a sample-path LDP for the model described in Section \ref{Sec:model}. In Sections ... we consider two applications of our main result. The numerical results are computed using the method described in Section ... .

\subsubsection{Sample-path LDP}
The next step is to extend the LDP for finitely many points in time to a full sample-path LDP 
{on $D(\mathbb{R}^2,[0,T])$, the space of $\mathbb{R}^2$-valued c\`{a}dl\`{a}g functions endowed with the Skorokhod topology}, with rate function $I_{[0,T]}(f,g)$ defined later in \eqref{eq:IsD}.
{Roughly speaking,} this is done in two steps: 
(i) we derive limiting expressions for $M^-_{\bs t}(\bs\omega, \bs\theta)$ and $M^+_{\bs t}(\bs\omega, \bs\theta)$ as the mesh $0=t_1 < t_2 < \dots < t_d=T$ becomes infinitely fine (done in Section \ref{Sec:SPLDP1}); (ii) we prove that the sequence of processes $(\bar F_n(\cdot), \bar G_n(\cdot))$ is exponentially tight (done in Section \ref{Sec:ExT}). As it turns out, from a computational perspective it is easier to work with a different  expression for the rate function $I_{[0,T]}(f,g)$: as pointed out in Section \ref{Sec:AF} we can decompose
$I_{[0,T]}(f,g)$ into two parts under the proviso that both $f$ and $g$ are absolutely continuous.
\iffalse
\[
I_{[0,T]}(f,g) = I_{[0,T]}(f) + I_{[0,T]}(g\,|\,f).
\]
The first term, $I_{[0,T]}(f)$, corresponds to the contribution to the rate function due to the client-population-size process, and is given by
\begin{align}\nonumber
 I_{[0,T]}(f)=\sup_{z(\cdot)} \bigg\{
 \log z(T) f(T)&-
 \int_0^T\log z(s) \,f'(s)\,{\rm d}s-f_0\log\left(\int_0^T h^\circ(u)\,z(u)\,{\rm d}u +\bar h^\circ(T)\,z(T)\right)\\
 &- \lambda\int_0^T\left(
 \int_s^T h(r-s)\,\frac{z(r)}{z(s)}\, {\rm d}r
 +  \bar h(T-s)\,\frac{z(T)}{z(s)} -1
 \right){\rm d}s
    \bigg\}, \nonumber
\end{align}
where we define $\bar h(t)=\int_t^\infty h(s)\,{\rm d}s$ and $\bar h^\circ(t)=\int_t^\infty h^\circ(s)\,{\rm d}s$.
The second term, $I_{[0,T]}(g\,|\,f)$, can be interpreted as the contribution of the net aggregate claim process conditional on the client-population-size path, and is given by
\[I_{[0,T]}(g\,|\,f) = \int_0^T K_{f(s)}(g'(s))\,{\rm d}s,\]
where 
\[K_x(u)= \sup_\theta(\theta u -x\,\varphi(\theta)).\]
\fi
On the other hand, when $f$ or $g$ is not absolutely continuous, then we show (also in Section \ref{Sec:ExT}) that $I_{[0,T]}(f,g)=\infty.$
Formally, our LDP result is summarized in the following statement. 

%As we will see below, this will involve an exponential tightness argumentation.
%In our sample-path LDP the role of the rate function of the path $(f,g)\equiv (g(s),f(s))_{s\in[0,T]}$ will be played by
%\begin{align*}
%I_{[0,T]}(f,g)&:=
%\sup_{\omega(\cdot),\theta(\cdot)} \bigg\{ \int^T_0 \left[ \theta(s)g(s) + \omega(s) f(s) \right]{\rm d}s\\ &\:\hspace{1cm}-\: f_0\log \mathbb{E} \exp \left( \int^T_0 \left[ \theta(s)G_-(s) + \omega(s)F_-(s) \right]{\rm d}s \right)  \\
%& \:\hspace{1cm}-\:\log \mathbb{E} \exp \left( \int_0^T \left[ \theta(s)G_+(s) + \omega(s)F_+(s) \right]{\rm d}s \right) \bigg\}.
%\end{align*}
%
%
%Importantly, the two mgf\,s that appear in this Legendre transform can be computed explicitly, as limits of the expressions for $M^-_{\bs t}(\bs \omega, \bs \theta)$ and $M^+_{\bs t}(\bs \omega, \bs \theta)$ that we derived in the previous section.  This we do in the first subsection, whereas in the second subsection we focus on the sample-path LDP.

%We let $D=D([0,T], (\mathbb{R}_{\geqslant 0})^2)$ be the space of cadlag paths in $(\mathbb{R}_{\geqslant 0})^2$ and we suppose that this space is equipped with a metric that induces the Skorokhod topology.

\begin{theorem}\label{thm:main}
The sequence of processes $(\bar F_n(t), \bar G_n(t))_{t \geqslant 0}$ satisfies the LDP {on $D(\mathbb{R}^2,[0,T])$} with rate $n$ and rate function $I_{[0,T]}(f,g)$ {characterised by \eqref{eq:IsD} and Lemmas \ref{L3} and \ref{L4}}.
\end{theorem}

\subsection{{Experiments}} Evidently, the primary application of Theorem \ref{thm:main} is to evaluate the decay rate of the time-dependent ruin probability $p_n(u,T)$, as was defined in \eqref{pnut}, in our model with a fluctuating number of clients. In addition, however, it reveals the most likely way in which rare events, such as the insurance firm going bankrupt, occur. 
In this subsection we apply Theorem \ref{thm:main} to explore in detail two areas of interest, both of them related to the most likely path to bankruptcy. 
\begin{itemize}
\item[(1)] What is the most likely path of $(\bar F_n(\cdot), \bar G_n(\cdot))$ to bankruptcy at some time $T$? More specifically, is the insurance firm more likely to go bankrupt when there are more clients than usual or fewer clients than usual? We remark that the answer to this question is not {\it a priori} obvious: more clients means more revenue, but also a higher risk of large claims. This question will be systematically analyzed in Section \ref{sec:bank}.
\item[(2)] What is the primary cause of fluctuations in the reserve level at some time $T$? Specifically, when are these fluctuations primarily due to randomness in the number of clients, and when are they primarily due to randomness in the claims made by these clients? We shed light on this issue in Section \ref{sec:fluc}.
\end{itemize}

As argued below, in the context of both questions, a crucial role is played by the probability that the process $( \bar F_n(\cdot),\bar G_n(\cdot))$ is in the `ruin set' $\mathscr{R}:= [0,\infty)\times B$ at time $t\geqslant 0$ \textcolor{black}{with $B \subset \mathbb{R}$}, i.e., the probability that $(\bar F_n(\cdot),\bar G_n(\cdot))$ belongs to 
\begin{equation}\label{VPA1}
\mathscr{H}_t := \{ (f,g) : (f(t), g(t)) \in \mathscr{R}\}.
\end{equation}
In light of Theorem \ref{thm:main}, in order to find the logarithmic decay rate of this probability we are  to solve the {variational problem}
\begin{equation}\label{VPA2}
\varrho(t):= \inf_{(f,g) \in \mathscr{H}_t} I_{[0,t]}(f,g).
\end{equation}
All numerical results included in this section are obtained using the method that is outlined in {Appendix} \ref{Sec:Comp}.

\subsubsection{{Path} to bankruptcy}\label{sec:bank}
We consider the situation that  $\mathscr{R} = [0,\infty)\times [u,\infty)$, where $u$ corresponds to the initial surplus of the insurance firm. This means that, due to Theorem \ref{thm:main}, the logarithmic decay rate of the time-dependent ruin probability can be found by solving the following optimisation:
\[\lim_{n\to\infty}\frac{1}{n}\log p_n(u,T) = - \inf_{t \in[0,T]} \varrho(t).\]
We start, however, by studying the probability of \emph{eventual} bankruptcy, i.e., bankruptcy over an infinite horizon (in the literature also frequently referred to as the {\it all-time ruin probability}). 
To this end, we consider
\[\lim_{n\to\infty}\frac{1}{n}\log p_n(u,\infty)=-\varrho\s := -\inf_{t \geqslant 0} \varrho(t).
\]
Let $t\s$ denote the corresponding optimising time (so that $\varrho\s = \varrho(t\s)$), and $f\s$, $g\s$ be the corresponding optimising paths (so that $\varrho\s = I_{[0,t\s]}(f\s,g\s)$).
%The next proposition states that when $\mathscr{H}$ takes the form in \eqref{HForm} the rate associated with eventually entering $\mathscr{H}$, $\varrho\s$, is independent of the parameters which control the queue length.
The next proposition reflects the remarkable fact that the probability of {eventual} bankruptcy is independent of fluctuations in the number of clients. An explanation of this fact is given below. Note that this result holds not only when $\mathscr{R}$ is of the form $[0,\infty) \times [u,\infty)$, but more generally when $\mathscr{R} = B \times [0,\infty)$ with $B \subset \mathbb{R}$.

 \begin{proposition}\label{PropTI1}
If $\mathscr{R}=B \times [0,\infty)$ with $B \subset \mathbb{R}$, then $\varrho\s$ is independent of the client-level dynamics $($i.e., $f_0$, $\lambda$, $h^\circ(\cdot)$, and $h(\cdot))$. {In addition, $\varrho\s$ only depends on $r$ and $\nu$ through the ratio ${r}/{\nu}$}. 
% \footnote{\small MM: I am not sure this is what you want to say. Isn't it that the decay rate depends on $r$ and $\nu$ through their ratio? PB: Yes, this is not correct. Perhaps we can just remove $r/\nu$.}
%In particular $\varrho\s$ is precisely the same as in the standard CL model (i.e., without allowing the number of customers to vary).
\end{proposition}

To understand the result stated in Proposition \ref{PropTI1}, it is instructive to compare the evolution of the net aggregate claim process $G_{n,1}(\cdot)$ when the client-population-size path is known to be $F_{n,1}(\cdot) = f(\cdot)$, to the evolution of the net aggregate claim process $G _{n,2}(\cdot)$ when the client-population-size path is (say) halved (i.e., it becomes $F _{n,2}(\cdot)=f(\cdot)/2$).
Because clients generate claims independently according to a Poisson process $\nu$ and generate capital at a constant rate $r$, we thus have \[G_{n,1}(t) \stackrel{{\rm d}}{=} G _{n,2}(2t)\] for all $t \geqslant 0$.
This means that an increase in the number of clients speeds up the evolution of the net aggregate claims  (which can be interpreted as time contraction), whereas a decrease in the number of clients slows down the evolution of the net aggregate claims (interpreted as time dilation). This local `compressing' or `stretching'  of time evidently has no impact on the probability of eventual bankruptcy. 
Thus, Proposition \ref{PropTI1} reflects the fact that the probability of \emph{eventual} bankruptcy is independent of any contraction/dilation of time.
{It is noted that the above arguments extend beyond our large deviation context, and therefore imply a more general property of the CL model with fluctuating client population. Indeed, in Proposition \ref{PropTI1} the decay rate $\varrho\s$ can be replaced by `the {all-time ruin probability}', again relying on the elementary time-contraction/time-dilation argumentation provided above.}

%  where the number of customers may vary, to a standard CL model, where the number of customers is fixed. For the purpose of this explanation suppose that in the standard CL model there is always exactly $n$ customers. In this case the number of customers in the queue effectively contracts/dilates time. For example, if the number of customers in the queue is $2n$ then the capital evolves as it would in the standard CL model but twice as fast (time contraction), whereas if the number of customers in the queue is $n/2$ then the capital evolves as it would in the standard CL model but at half the speed (time dilation). Thus, Proposition \ref{PropTI1} reflects the fact that the probability of \emph{eventual} bankruptcy is independent of any contraction/dilation of time.

Where Proposition \ref{PropTI1} concerns  the all-time ruin probability, in applications  one is, for obvious reasons, typically interested in the  time-dependent ruin probability, i.e., the probability of the insurance firm being bankrupt {by a given time $T>0$}. 
Importantly, in this case fluctuations in the number of clients {\it do} play an important role in determining the probability of bankruptcy.
As we will show now, however, we can use the ideas that underlie Proposition~\ref{PropTI1}  to {identify some structural properties corresponding to} this finite-horizon context, too.
% However, we are able to extend the line of reasoning that led to the previous result on the probability of eventual bankruptcy to establish a result in this context, too.
To this end, let $(f^{(\star,T)}, g^{(\star,T)})$ be the most likely path in $\mathscr{H}_T$, so that $\varrho(T)=I_{[0,T]}( f^{(\star,T)},g^{(\star,T)})$. 
In addition, let $(\bar f, \bar g)$ 
satisfy
%the unique\footnote{\small MM: how do we know it's unique? PB: Convexity of the 1-D rate functions (since they we established through Cramer's theorem) and strict convexity at the mean (since there is a finite variance). I don't think we should mention this.} pair that satisfies 
$I_{[0,\infty)}(\bar f, \bar g)=0$, so that $(\bar f, \bar g)$ can be interpreted as the fluid limit of $(\bar F_n(\cdot),\bar G_n(\cdot))$.

{First consider the case that the horizon $T$ equals the most likely time $t\s$ of  eventual bankruptcy. In view of the argumentation underlying Proposition \ref{PropTI1}, one anticipates that the client population evolves (most likely) along its fluid-limit path:
\begin{itemize}
    \item[\emph{(i)}] $f\s(t):=f^{(\star,t\s)}(t)=\bar f(t)$ for all $t \in [0,t\s]$.
\end{itemize}
Next suppose $T<t\s$. In this case the process $\bar G_n(\cdot)$ must enter the rectangular set $\mathscr{R}$ faster than it would do in the infinite-horizon case. 
In the most likely path, one thus anticipates that the number of clients is higher than expected in order to speed up the evolution of $\bar G_n(\cdot)$.
%and thus move toward the optimal timescale. 
This reasoning leads to
\begin{itemize}
    \item[\emph{(ii)}] if $T < t\s$ then $f^{(\star,T)}(t) > f\s(t)=\bar f(t)$ for all $t \in [0,T]$.
\end{itemize}
Similarly, if $T>t\s$, then $\bar G_n(\cdot)$ must enter $\mathscr{R}$ more slowly than it  would  optimally do. In the most likely path, one anticipates the number of clients to be lower than expected in order to slow down the evolution of $\bar G_n(\cdot)$, i.e.,
\begin{itemize}
\item[\emph{(iii)}] if $T > t\s$ then $f^{(\star,T)}(t)<f\s(t)=\bar f(t)$ for all $t \in [0, t\s]$.
\end{itemize}
}
 
% \begin{proposition}\label{PropTI2}
% If $\mathscr{H}=B \times [0,\infty)$ with $B \subset \mathbb{R}$, then $f\s(t)=\bar f(t)$ for all $t \in [0, t\s]$. In addition,
% \begin{itemize}
% \item[\emph{(i)}] if $T < t\s$ then $f^{(\star,T)}(t) > f\s(t)=\bar f(t)$ for all $t \in [0,T]$
% \item[\emph{(iii)}] if $T > t\s$ then $f^{(\star,T)}(t)<f\s(t)=\bar f(t)$ for all $t \in [0, t\s]$.
% \end{itemize}
% \end{proposition}
% The first assertion in Proposition \ref{PropTI2} states that the path of the client number conditioned on eventual bankruptcy of the company coincides with its fluid limit.
% This reflects the fact that there is a cost associated with either contracting or dilating time (i.e., having an unsually large or small number of clients), therefore in the most likely path to eventual bankruptcy time proceeds as expected. 
% The second assertion in Proposition \ref{PropTI2} reflects the fact that if $T<t\s$ then the process needs to enter the rectangular set $\mathscr{H}$ faster than would be optimal, and hence the client number increases in order to speed up the evolution of $G_n(t)$ (contract time) in an effort to move toward the optimal timescale. 
% Similarly, if $T>t\s$ then the process needs to enter $\mathscr{H}$ slower than would be optimal, and hence the client number decreases in order to slow down the evolution of $G_n(t)$ (dilate time) in an effort to move toward the optimal timescale.

\begin{figure}
\includegraphics[width=5.4cm]{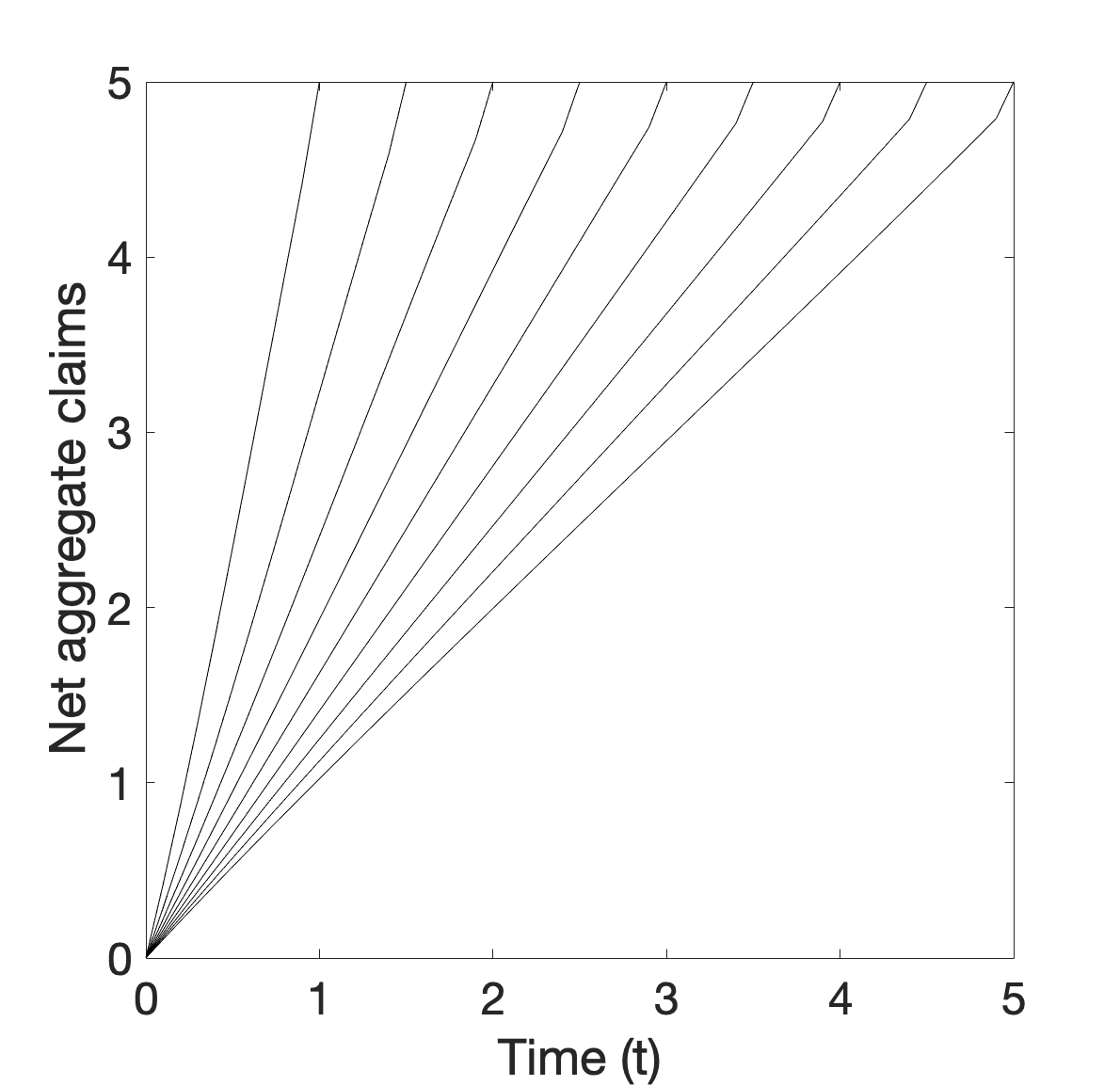}
\includegraphics[width=5.4cm]{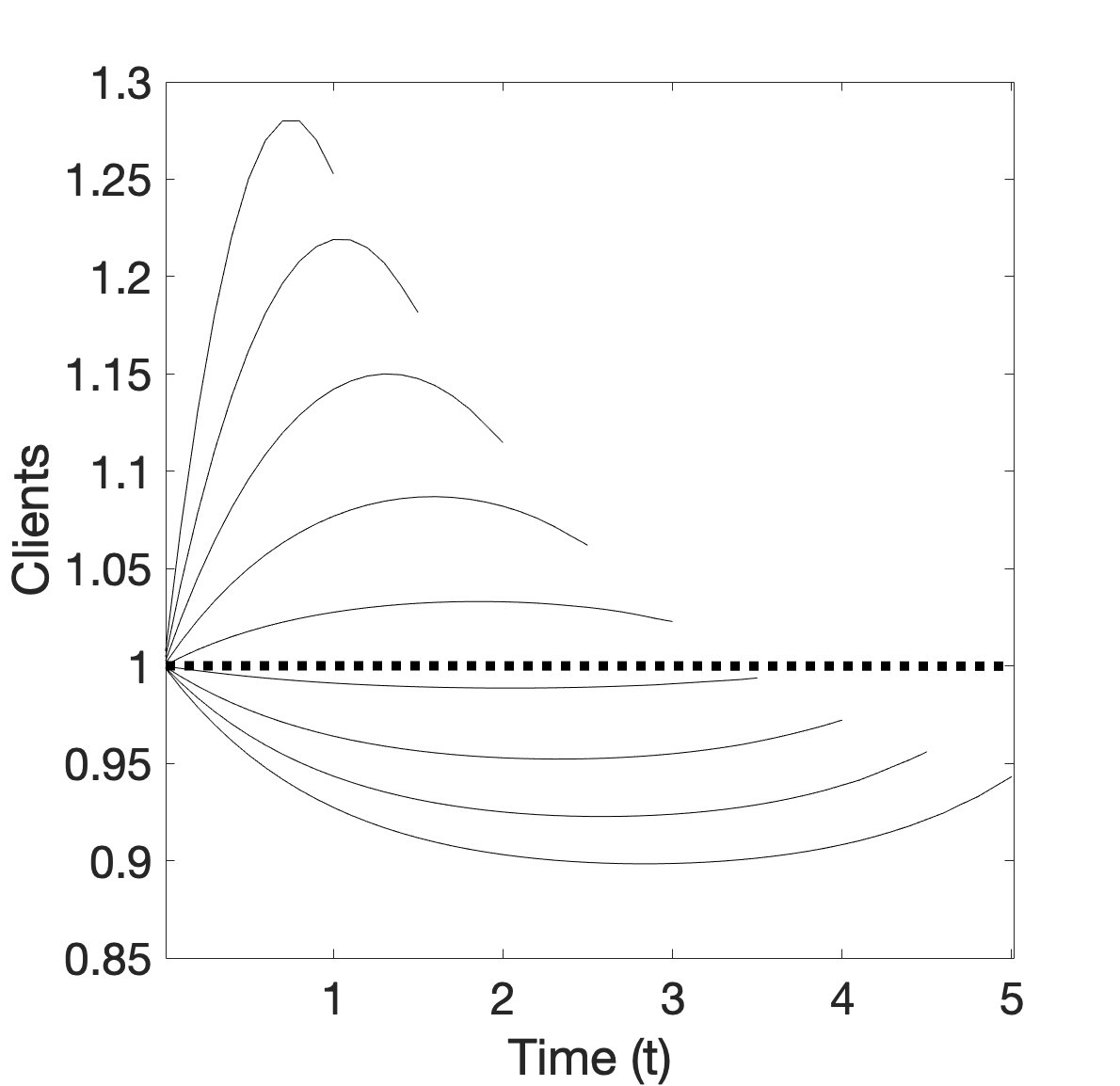}
\includegraphics[width=5.4cm]{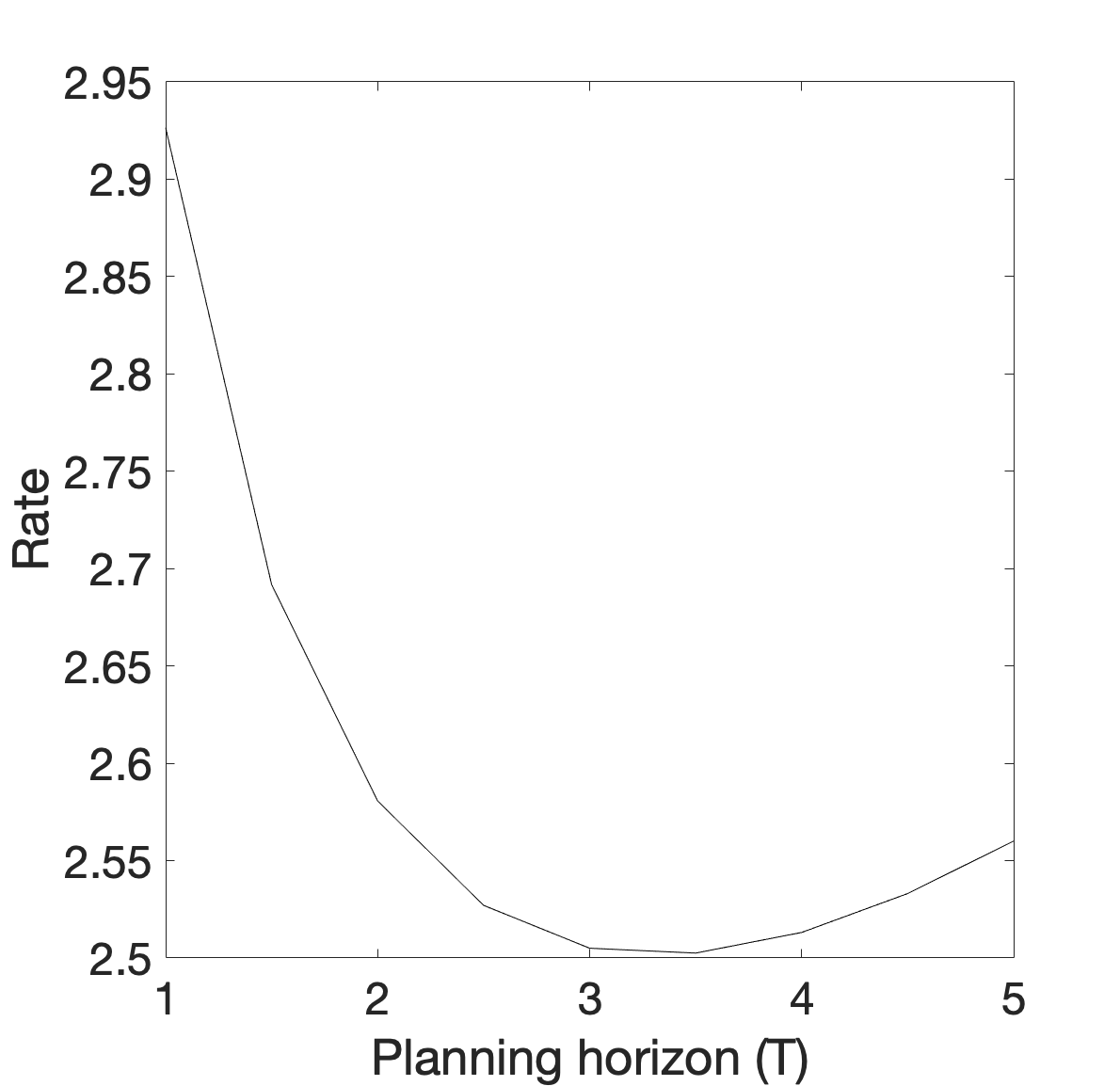}
\includegraphics[width=5.4cm]{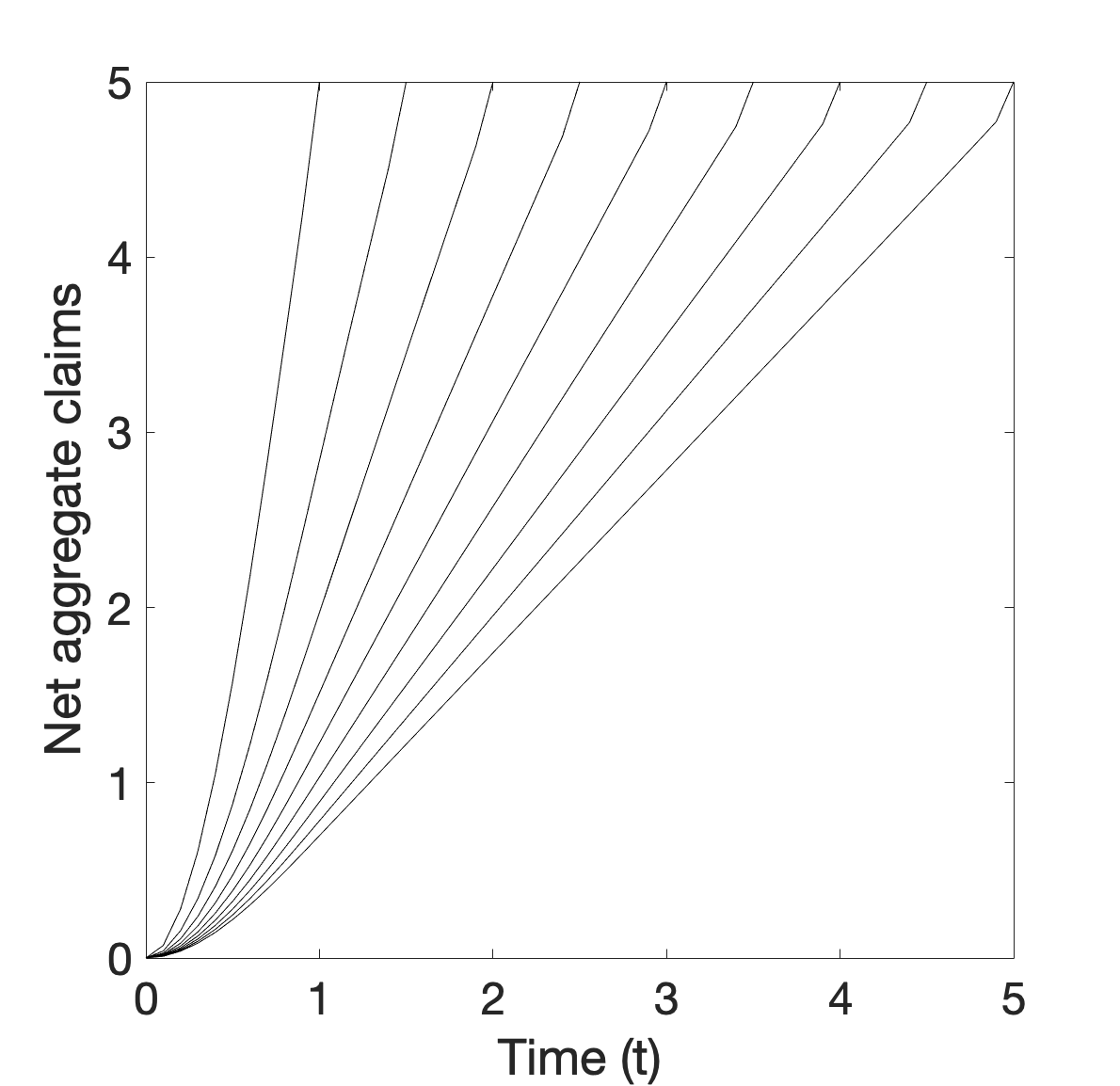}
\includegraphics[width=5.4cm]{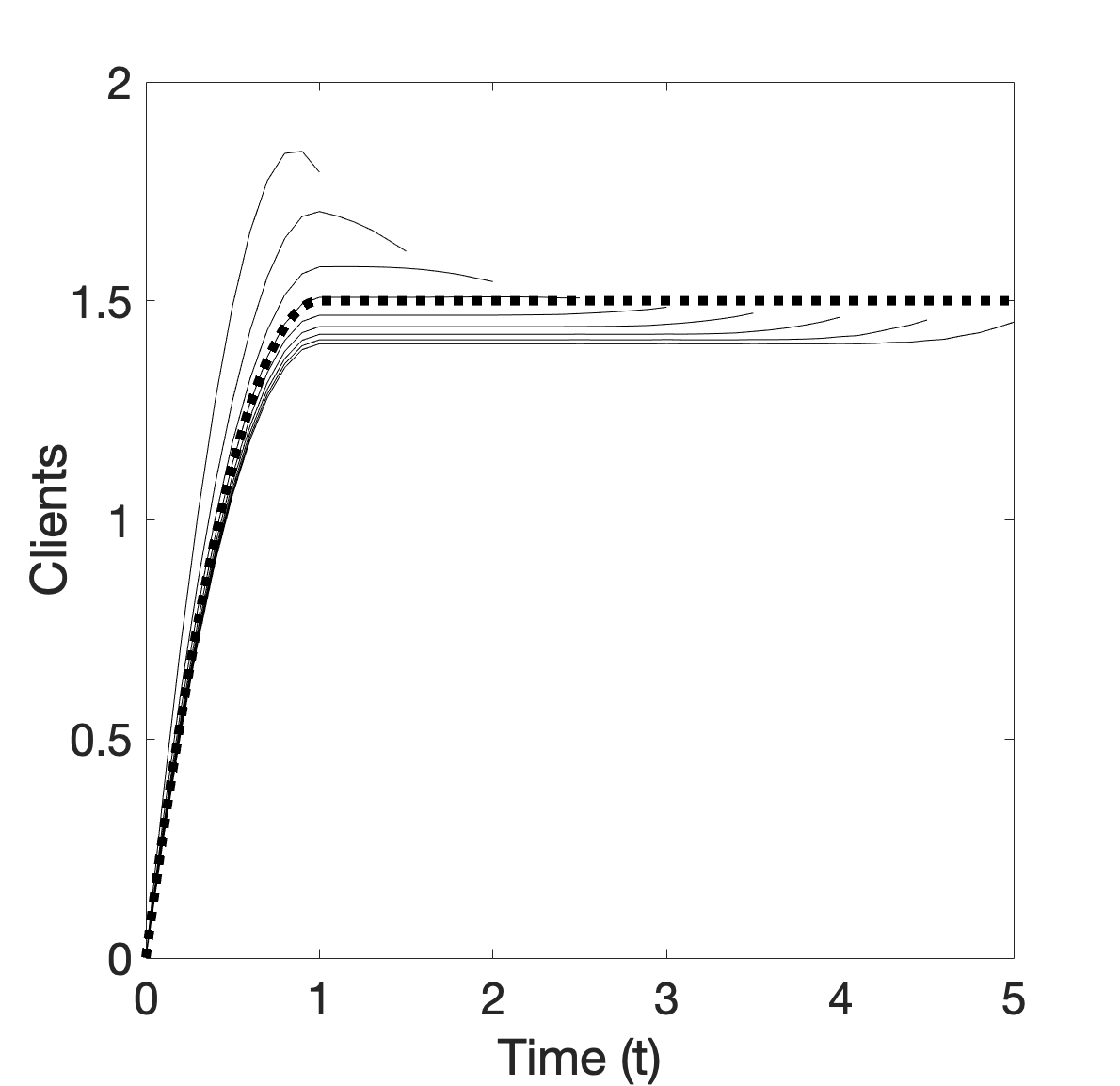}
\includegraphics[width=5.4cm]{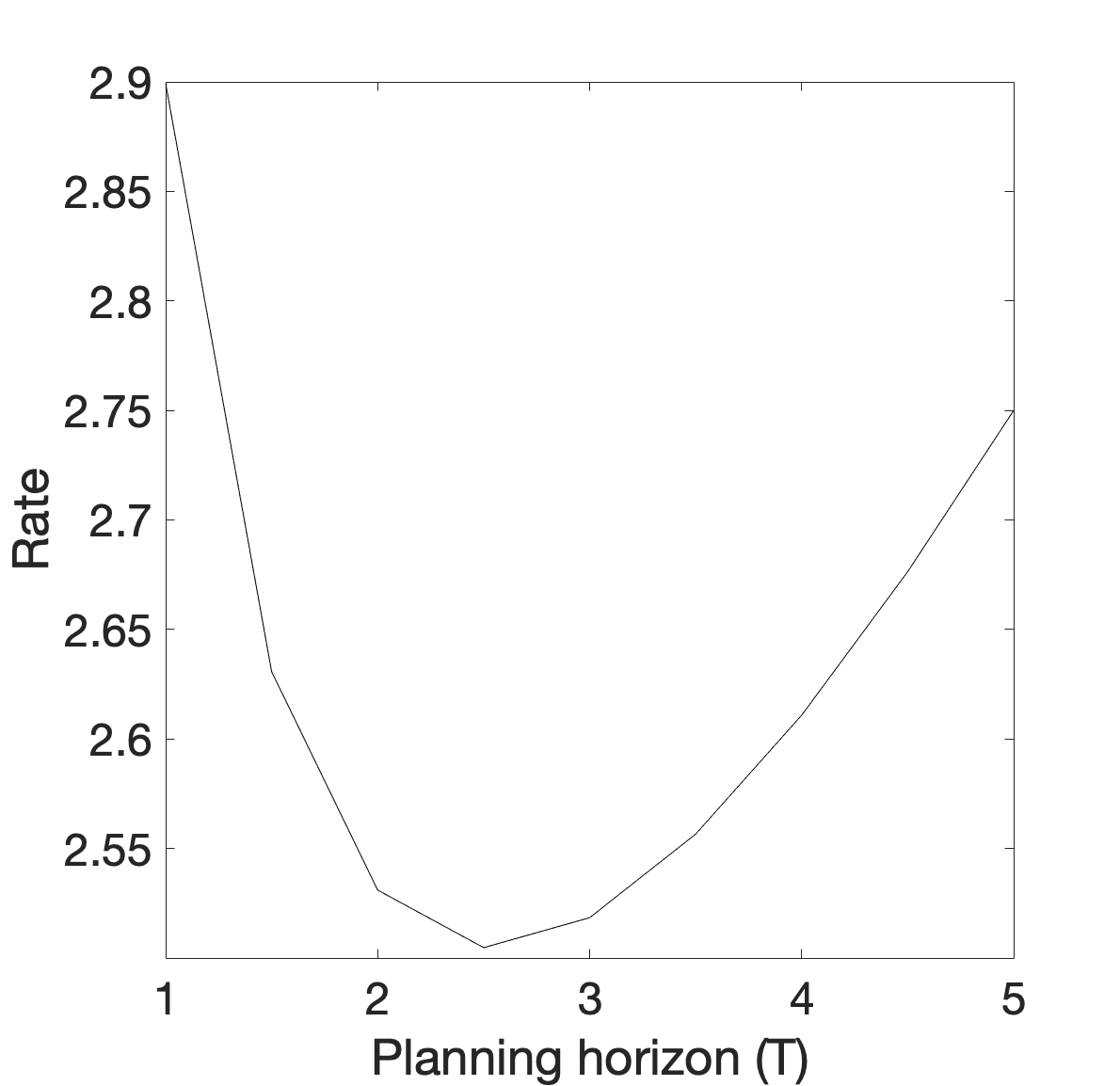}
\caption{\label{Fig:QPs} Most likely paths to bankruptcy for various values of the time horizon $T$. Left panels: most likely net aggregate claim path $g^{(\star,T)}(\cdot)$, middle panels: most likely client-population-size path $f^{(\star,T)}(\cdot)$), right panels: the corresponding decay rate $I_{[0,T]}(g^{(\star,T)},f^{(\star,T)})$. The top and bottom panels correspond to different parameter values, which are given in the text. The dashed curve in the middle panels is $\bar f(\cdot)$.}
\end{figure}

% We illustrate Propositions \ref{PropTI1} and \ref{PropTI2}
\begin{example} We further study the properties {\it (i)--(iii)}
by means of two numerical \textcolor{black}{experiments} that are pictorially illustrated in Figure~\ref{Fig:QPs}. In both \textcolor{black}{experiments} the net aggregate claim process is characterised by $\nu=3$, $r=3$, with the claim sizes being  exponentially distributed with mean~$\frac{2}{3}$. The insurance firm initially has five units of capital (i.e., $u=5$), and we consider time horizons $T\in\{1,1.5, \dots, 5\}$. 
In the top row of Figure \ref{Fig:QPs} we let $f_0=1$, $\lambda=1$, and the sojourn-time distribution be exponential with mean~$1$ (with the residual sojourn times of the clients present at time $0$ being exponential with mean 1 as well). Note that this means $\bar f(t)=1$ (dashed curve) for all $t\geqslant 0$; informally, the population-size process starts in equilibrium. In the bottom row of Figure \ref{Fig:QPs} we let $f_0=0$, $\lambda=3$, and the sojourn-time distribution be uniform on $[0,1]$ (with $h^\circ(\cdot)$ being the corresponding residual distribution).
%We plot the conditioned net aggregate claims $g^{(\star,T)}$ (left column), conditioned client number $f^{(\star,T)}$ (center column), and the corresponding rate $I_{[0,T]}(g^{(\star,T)},f^{(\star,T)})$ (right column). 

Observe that the time horizon associated with the minimal decay rate (i.e., the most likely timescale of ruin in the infinite-horizon case) is $T=3.5 \approx t\s$ for the parameter values in the top row, and $T=2.5  \approx t\s$ for the parameter values in the bottom row. In the right panels of Figure \ref{Fig:QPs} we see that, in line with Proposition \ref{PropTI1}, the rates associated with these optimal time horizons are equal (with $\varrho\s \approx 2.5$). 
In addition, in the center column of Figure \ref{Fig:QPs} we see that, corroborating the properties \emph{(i)}--\emph{(iii)} above, the conditioned path of the clients $f^{(\star,T)}(\cdot)$ is larger than $\bar f(\cdot)$ (depicted by the dashed curve) when $T<t\s$, smaller  than $\bar f(\cdot)$ when $T>t\s$, and equal to $\bar f(\cdot)$ when $T=t\s$.
\end{example}

 \subsubsection{What is the primary cause of fluctuations in capital: clients or claims?}\label{sec:fluc}
 
%\footnote{\small MM: We should avoid using both $C$ and $m_C$. This is confusing as one of the $C$\,s is a mnemonic for claims, and the other for capital. That's why I write $\bar m$.} 
We suppose that {a} {\it net profit condition} is in place, i.e., we are in the situation  that  $r>{\bar m} \nu$, where ${\bar m}$ is the expected value of the claim size. This condition is natural as it entails that, on average, each client generates a positive return for the insurance company. 
{Our objective is to understand the most likely cause of unusual
values of the net aggregate claim process.
Evidently this is connected to ruin, as ruin occurs when the net aggregate claim process is unusually large.
However, for ease of exposition we start by considering the case that the net aggregate claim process attains an unusually {\it small} value $a$ at time $T$ (corresponding to an unusually large value of the surplus process).}

% The insurance firm obviously aims at avoiding ruin, corresponding to paths in which the net aggregate claim process is unusually large. However, 
% also the event that the net aggregate claim process attains unusually small values is of interest -- this scenario frequently happening may be an indication that the premium rate could be lowered\footnote{MM: Is this reasoning OK?}. 

% In what follows our objective is to understand the most likely cause of unusual
% values of the net aggregate claim process. In our experiments we focus on the case 
% that the net aggregate claim process attains an unusually small value $a$ at time $T$ (corresponding to an unusually large value of the surplus process).\footnote{MM: Can we do something similar for unusually {\it large} values $a$?} To this end, we let $a<\bar g(T)$, where we recall that $\bar g(\cdot)$ is the fluid limit corresponding to the net aggregate claim process $G_n(\cdot)$. 
Suppose $a<\bar g(T)$, where we recall that $\bar g(\cdot)$ is the fluid limit corresponding to the net aggregate claim process $G_n(\cdot)$.
One could distinguish between two possible causes for a large surplus to happen. Contribution~(1) reflects  the event that the number of clients that the insurance company attracts is larger than one would expect; due to the net profit condition this scenario corresponds  to a higher surplus.
% \footnote{MM: You wrote `the insurance company attracts more clients than expected', but this is a strange reasoning, in light of the earlier remark that more clients also leads to more claims. That's why I explicitly added the net profit stuff.} 
Contribution~(2) reflects the event that the {client}-population size attains its expected value but the amount of money claimed by the clients present is lower than expected. Our objective is to quantify the  Contributions (1) and (2). To determine the proportion of the additional capital $\bar g(T)-a$ that can be attributed to additional clients (i.e., Contribution 1) we introduce the performance metric
\begin{equation}\label{CC1}
E_1({a},T) := \frac{(r-{\bar m} \nu)\int_0^T [f^{(\star,T)}(t)-\bar f(t)]\, {\rm d}t}{\bar g(T)-a},
\end{equation}
with, as before, $\bar f(\cdot)$ denoting the fluid limit of the process $F_n(\cdot)$.
Observe that the numerator of \eqref{CC1} can be interpreted as the additional clients $f^{(\star,T)}(t)-\bar f(t)$ in the most likely path multiplied by the {expected net rate} $r-{\bar m} \nu>0$ that these clients generate capital, integrated over time. We divide by the total additional capital {$\bar g(T)-a$} to obtain a proportion. 
What remains can be attributed to clients generating fewer claims than expected (i.e., Contribution 2),
\[
E_2(a,T) := 1-E_1(a,T).
\]
In this way we have separated the effect due to the fluctuations in the number of clients on one hand, and the effect due to the fluctuations in the amount of money claimed by the clients on the other hand.

{Under the net profit condition, when $a < \bar g(T)$ we expect that $E_1(a,T) > 0$ (and hence $E_1(a,T) \in [0,1]$). This is due to the time contraction/dilation arguments in Section \ref{sec:bank}.
In particular, if $a < \bar g(T)$ then there exists $t'>0$ such that $a=\bar g(T+t')$, so that, in order to move toward the optimal time scale, time should contract and hence we expect that $f^{(\star,T)}(t)>\bar f(t)$ for all $t \in [0,T]$. The same reasoning holds when $a \in [\bar g(T),0]$, although now with $t'<0$ and time dilation; however it breaks down when $a>0$, and in this case we may have $E_1(a,T)<0$.}
% because additional clients should lead to a greater chance of an unusually large surplus; however, when $a >\bar g(T)$ we may have $E_1(a,T)<0$ because, as we saw in Section \ref{sec:bank}, additional clients can increase the likelihood of an unusually high net aggregate claim.}

While $E_1(a,T)$ and $E_2(a,T)$ can be computed numerically, in general, it is challenging to express them analytically.
However, from Theorem \ref{thm:main} and elementary (but lengthy) calculations we can derive an expression \textcolor{black}{as $a-\bar g(T) \to 0$.} These calculations, sketched in Appendix \ref{app1}, 
involve equating the reward (gain in capital) per unit cost (increase in the rate function) for increasing the number of clients that arrive at any time $t$, and decreasing the {value} of claims generated by the clients. 
In particular, when the sojourn-time distribution is exponential with rate $\mu$, we obtain
\begin{equation}\label{ECT}
\lim_{a \to \bar g(T)} E_1(a,T) = \frac{\int_0^T (\lambda + \bar f(t) \mu) \left[ \frac{r-\nu \beta'(0)}{\mu} (1-e^{-\mu(T-t)}) \right]^2{\rm d}t }{\beta''(0) \nu \int^T_0 \bar f(t)\, {\rm d}t+\int_0^T (\lambda + \bar f(t) \mu) \left[ \frac{r-\nu \beta'(0)}{\mu} (1-e^{-\mu(T-t)}) \right]^2{\rm d}t}.
\end{equation}
The individual expressions appearing in the right-hand side of \eqref{ECT} have the following interpretations. In the first place, $(\lambda + \bar f(t) \mu)\,{\rm d}t$ is proportional to the variance of the difference in the number of clients at $t$ and $t+{\rm d}t$, respectively, when the number of clients at time $t$ is close to the fluid limit $\bar f(t)$. Secondly,
\[\frac{r-\nu \beta'(0)}{\mu} (1-e^{-\mu(T-t)})\] 
is the expected capital that is earned from a single client that arrives at time $t$. Thirdly, \[\beta''(0) \nu \int^T_0 \bar f(t) \,{\rm d}t\] is proportional to the variance in the total value of claims when the number of clients follows its fluid limit $\bar f(\cdot)$.
In view of the above, $\lim_{a \to \bar g(T)} E_1(a,T)$ has the appealing interpretation of a ratio of variances. 
A similar  expression with the same interpretation can be obtained when considering the case with general  sojourn times. {This expression, being considerably more involved, is left out.} 

\begin{figure}
\includegraphics[width=13.5cm]{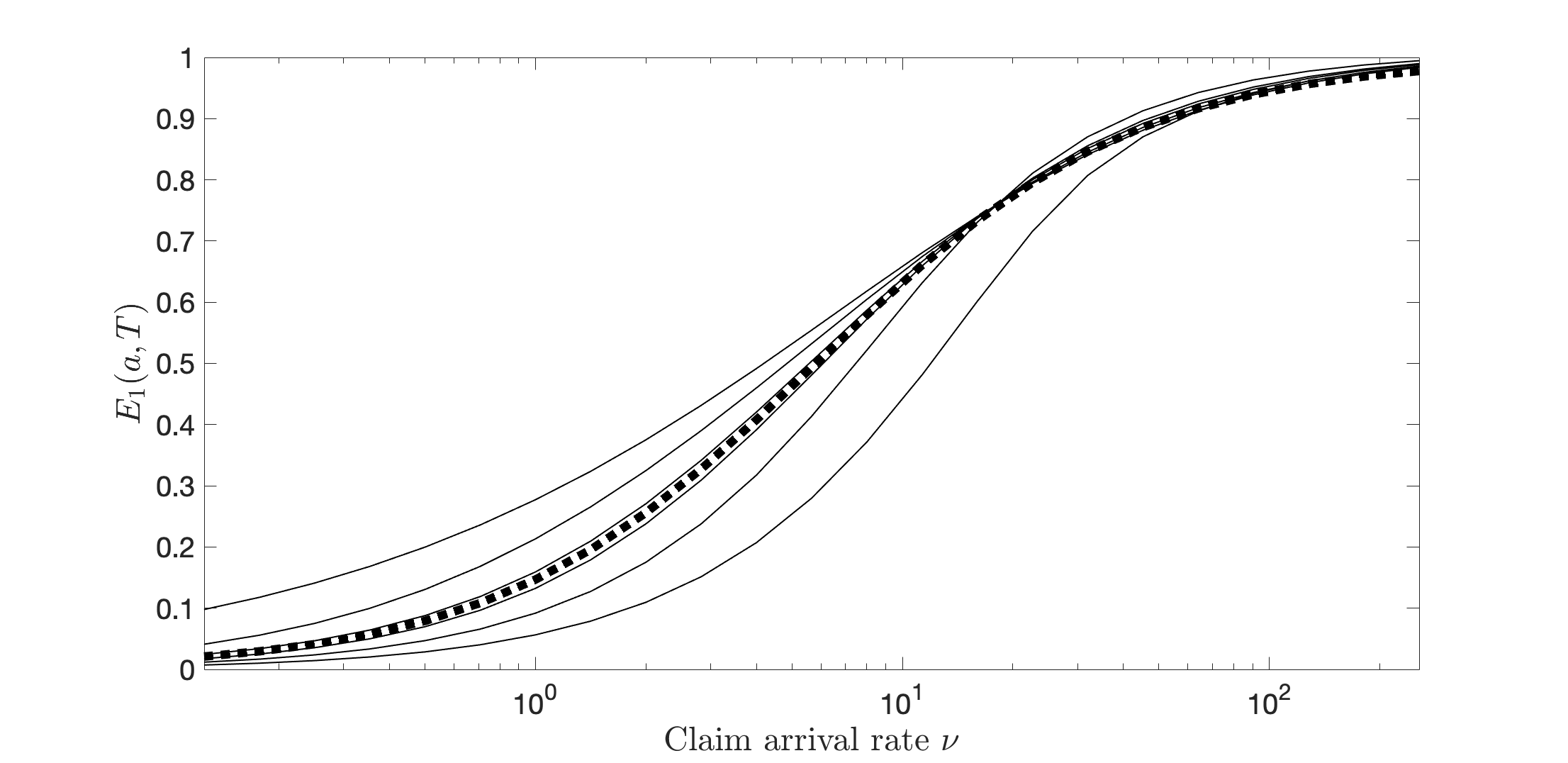}
\caption{\label{CPE}The proportional effect of fluctuations in the number of clients: $E_1(a,T)$ as a function of the claim arrival rate $\nu$.}
\end{figure}

\begin{example} We illustrate the concepts introduced above by means of a numerical example. We let $\lambda=1$, $f_0=1$, $T=1$, and suppose clients leave the company at rate $\mu=1$  (where it is noted that this implies that $\bar f(t)=1$ for all $t\in[0,T]$). Regarding the claim arrival process, we let $r=2$, and suppose that the claim sizes are exponentially distributed with a mean ${\bar m}$ such that
$\nu {\bar m}=1.$
%The condition $\nu {\bar m}=1$ means that the expected capital generated by each client is independent of $\nu$.\footnote{\small MM: This sentence I don't understand. It is of course true, but why is it mentioned?}
{Observe that $\bar g(T)=(\bar m \nu-r)\int^T_0 \bar f(t) {\rm d}t =-1$.}
%\footnote{\small MM: This I don't immediately see. PB: I added an additional step.}
In Figure \ref{CPE}
% \footnote{MM: change the $C$ into $a$, as $C$ is the upper bound on the density $h(\cdot)$. And write `claim arrival rate' rather than `Claim rate'.} 
we take values of $\nu$ ranging from $2^{-3}$ to $2^8$ and plot $E_{1}(a,T)$ for {$a=0$ (given by the lowest solid curve), $-0.5$, $-0.9$, $-1.1$, $-1.5$, $-2$ (given by the highest solid curve)}, and we plot the limiting value~\eqref{CC1} (given by the dashed curve). 
The figure illustrates that when $\nu$ is very large (and hence  ${\bar m}$ very small), then $E_1(a,T)$ is close to 1. 
This reflects the fact that, under $\nu {\bar m}=1$, \textcolor{black}{as} $\nu \to \infty$, each client generates claims in an increasingly deterministic manner, and hence large fluctuations in the capital are more likely to be caused by fluctuations in the number of clients. 
Evidently, the opposite reasoning applies when $\nu \downarrow 0$. 
The figure also shows that, in the setting considered, for lower values of $a$, fluctuations in the number of clients play an increasingly important role.
\end{example}

\section{Proofs of the large-deviation {principles} }\label{Sec:proofs}

\subsection{Finite-dimensional LDP}\label{Sec:FDLDP}

To establish the {multi-point} LDP stated in Proposition \ref{P2} it remains to compute $M^-_{\bs t}(\bs\omega, \bs\theta)$ and $M^+_{\bs t}(\bs\omega, \bs\theta)$ (as defined in \eqref{eq:Mdef}).
We start by evaluating $M^-_{\bs t}(\bs\omega, \bs\theta)$.

%We proceed by considering an LDP corresponding to times $0< t_1\leqslant t_2\leqslant \ldots \leqslant t_d$ for $d\in{\mathbb N}$. As before, we split the required mgf into one representing the contribution of the clients present at time $0$ and another corresponding 
%to the contribution of the clients arriving in $(0,t]$, with these two contributions being independent. As a consequence, by applying the $2d$-variate version of Cram\'er's theorem, the large deviations rate function is
%\[I_{\bs t}({\bs g},{\bs f})=\sup_{\bs \theta,\bs \omega} \left(\sum_{j=1}^d \theta_j g_j+\sum_{j=1}^d \omega_j f_j - f_0\,\log M^-_{\bs t}(\bs\omega, \bs\theta) - \log M^+_{\bs t}(\bs\omega, \bs\theta)\right),\]
%where \[M^i_{\bs t}(\bs\omega, \bs\theta) =  {\mathbb E}\exp\left(\sum_{j=1}^d \theta_j G^i(t_j)+
%\sum_{j=1}^d \omega_j F^i(t_j)
%\right),\] for $i\in\{-,+\}$.  Analogously to the procedure followed in the single-dimensional case, we derive expressions for the mgf\,s $M^-_{\bs t}(\bs\omega, \bs\theta)$ and $M^+_{\bs t}(\bs\omega, \bs\theta)$ separately.

%\textcolor{black}{Here it looks like we are implicitly assuming $f_0=1$. I guess we are instead computing $\Lambda^-_{\bs {t}}(\bs\theta)$.} 
Recall that $\tau^\circ$ is a variable corresponding to a typical residual sojourn-time duration of a client who is present at time $0$. {Let
$\Omega_k:=\sum_{j=1}^{k}\omega_j$.} With $t_0\equiv 0$ and, as before,  $\bar h^\circ(t):=\int_t^\infty h^\circ(s)\,{\rm d}s$,
%\footnote{PB: Should this be $\Omega_k:=\sum_{j=1}^k \omega_j$?}
\begin{align}
M^-_{\bs t}(\bs\omega, \bs\theta) &=
{\mathbb E}\exp\left( \sum_{j=1}^d 
\omega_j 1\{\tau^\circ >t_j\}+\sum_{j=1}^d\theta_j 
Z(\tau^\circ\wedge t_j)
\right)\nonumber \\ \label{E1}
&=\sum_{k=1}^d \int_{t_{k-1}}^{t_k} h^\circ(s) \,{\mathbb E}\exp\left(\sum_{j=1}^d\theta_j 
Z(s\wedge t_j)\right)\exp\left(\Omega_{k-1}\right){\rm d}s\:+\\&\:\:\:\: \:\:
\bar h^\circ(t_d)\,\, {\mathbb E}\exp\left(\sum_{j=1}^d\theta_j Z(t_j)\right)\exp\left(\Omega_d\right). \label{E2}\end{align}
To further evaluate $M^-_{\bs t}(\bs\omega, \bs\theta)$, let us first focus on Expression \eqref{E2}. By using a telescopic sum representation, and denoting $\Theta_k:=\sum_{j=k}^d\theta_j$ and $\delta_k:=t_k-t_{k-1}$, we obtain that the mgf featuring in this term equals
\begin{align*}
{\mathbb E}\exp\left(\sum_{j=1}^d\theta_j \sum_{k=1}^j\big(Z(t_k)-Z(t_{k-1})\big)\right)&=
{\mathbb E}\exp\left(\sum_{k=1}^d\Theta_k \big(Z(t_k)-Z(t_{k-1})\big)\right)=\prod_{k=1}^d(\varphi(\Theta_k))^{\delta_k}.
\end{align*}
The other term in $M^-_{\bs t}(\bs\omega, \bs\theta)$, i.e., Expression \eqref{E1}, can be computed along the same lines. To this end, we use that, evidently, for $s\in[t_{k-1},t_k)$ we have that $t_j\wedge s = t_j$ for $j = 0,\ldots,k-1$, whereas $t_j\wedge s = s$ for $j=k,\ldots,d.$ By some standard algebra, we thus obtain that the mgf featuring in the $k$-th term in the sum, for $s\in[t_{k-1},t_k)$,
\begin{align*}{\mathbb E}\exp\left(\sum_{j=1}^{k-1}\theta_j 
Z(t_j) +Z(s)\sum_{j=k}^{d}\theta_j 
  \right) = \left(\prod_{j=1}^{k-1} (\varphi(\Theta_j))^{\delta_j}\right)\cdot (\varphi(\Theta_k))^{s-t_{k-1}}
\end{align*}
Upon combining the above, we have found the following expression for $M^-_{\bs t}(\bs\omega, \bs\theta)$. Observe that it is fully in terms of the `partial sum series' $\Theta_k$ and $\Omega_k$, corresponding the arguments ${\bs \theta}$ and ${\bs \omega}$, respectively. 
\begin{lemma} \label{L1} For $\bs\omega, \bs\theta\in{\mathbb R}^d$,
\begin{equation}\label{Mme}
M^-_{\bs t}(\bs\omega, \bs\theta)= 
\sum_{k=1}^d e^{\Omega_{k-1}}\prod_{j=1}^{k-1} (\varphi(\Theta_j))^{\delta_j}\int_{t_{k-1}}^{t_k} h^\circ(s) \,
\varphi(\Theta_k)^{s-t_{k-1}}
{\rm d}s +
e^{\Omega_d}\,\bar h^\circ(t_d) \prod_{k=1}^d(\varphi(\Theta_k))^{\delta_k}.
\end{equation}
\end{lemma}

We now evaluate $M^+_{\bs t}(\bs\omega, \bs\theta)$. This is done by distinguishing the contributions due to clients who arrive in each of the intervals $[t_{\ell-1},t_\ell)$, for $\ell=1,\ldots,d$, which are independent (as argued before). We thus arrive at the decomposition
\begin{equation}\label{Mp}
M^+_{\bs t}(\bs\omega, \bs\theta) = \prod_{\ell=1}^d M^+_{{\bs t}, \ell}(\bs\omega, \bs\theta),
\end{equation}
where
\begin{equation}\label{Mpl}
M^+_{{\bs t}, \ell}(\bs\omega, \bs\theta) = \sum_{k=0}^\infty 
e^{-\lambda \delta_\ell}\frac{(\lambda \delta_\ell)^k}{k!} \left(\bar M^+_{{\bs t}, \ell}(\bs\omega, \bs\theta)\right)^k=\exp\left(\lambda\delta_\ell\left(\bar M^+_{{\bs t}; \ell}(\bs\omega, \bs\theta)-1\right)\right),
\end{equation}
here $\bar M^+_{{\bs t}, \ell}(\bs\omega, \bs\theta)$ is the mgf that corresponds to the contribution of a single client arriving at a uniformly distributed epoch in $[t_{\ell-1},t_\ell)$, an interval of length $\delta_\ell$. 
As a consequence, with $t_{j,\ell}:=t_j-t_\ell$, we have that
\[\bar M^+_{{\bs t}, \ell}(\bs\omega, \bs\theta) = 
{\mathbb E}\exp\left(\sum_{j=\ell}^d \theta_j Z(\tau\wedge (\delta_\ell(1-U)+t_{j,\ell}))+
\sum_{j=\ell}^d \omega_j 1\{\tau > \delta_\ell(1-U)+t_{j,\ell}\}
\right).\]
By distinguishing between the values of $\tau$, the expression in the previous display can be decomposed into the sum of the three terms. The first term
corresponds with the scenario that the client has left by time $t_\ell$. It can be written as
\begin{align}
\nonumber \bar M^+_{{\bs t}, \ell,1}(\bs\omega, \bs\theta) &:=\int_0^{\delta_\ell}\frac{1}{\delta_\ell}\int_0^{\delta_\ell-s}h(r)\, {\mathbb E}\exp\left(\sum_{j=\ell}^d \theta_j Z(r)\right){\rm d}r\,{\rm d}s
\\&=
\int_0^{\delta_\ell}\frac{1}{\delta_\ell}\int_0^{t_{\ell,\ell-1}-s}h(r)\, (\varphi(\Theta_\ell))^r{\rm d}r\,{\rm d}s \label{Mbpl1}
.\end{align}
The second term
corresponds to the scenario that the client has left between $t_k$ and $t_{k+1}$, for some index $k\in\{\ell,\ldots,d-1\}$. We obtain
\begin{align}\nonumber\bar M^+_{{\bs t}, \ell,2}(\bs\omega, \bs\theta) &:=\int_0^{\delta_\ell}\frac{1}{\delta_\ell}
\sum_{k=\ell}^{d-1}\int_{t_{k,\ell-1}-s}^{t_{k+1,\ell-1}-s}h(r)\cdot\\
&\hspace{3cm}
{\mathbb E}\exp\left(\sum_{j=\ell}^{k} \theta_j Z(t_{j,\ell-1}-s)
+\sum_{j=k+1}^d \theta_j Z(r)+\sum_{j=\ell}^{k} \omega_j
\right){\rm d}r\,{\rm d}s \nonumber \\
&\hspace{-2mm}=\int_0^{\delta_\ell}\frac{1}{\delta_\ell}
\sum_{k=\ell}^{d-1}\int_{t_{k,\ell-1}-s}^{t_{k+1,\ell-1}-s}h(r)\,e^{\Omega_k-\Omega_{\ell-1}}\, (\varphi(\Theta_\ell))^{\delta_{\ell}-s}\cdot\nonumber\\
&\hspace{3cm}
\left(\prod_{m=\ell+1}^k (\varphi(\Theta_m))^{\delta_m}\right)
\cdot((\varphi(\Theta_{k+1}))^{r-(t_{k,\ell-1}-s)} {\rm d}r\,{\rm d}s\label{Mbpl2}
.\end{align}
To verify the expression \eqref{Mbpl2} in the above display, use the distributional equality,
with $Z(s,t):=Z(t)-Z(s)$ for $s\leqslant t$,
\begin{align*}
    \sum_{j=\ell}^{k}& \theta_j Z(t_{j,\ell-1}-s)
+\sum_{j=k+1}^d \theta_j Z(r) \\
&\stackrel{\rm d}{=}\Theta_\ell Z(t_{\ell-1}+s,t_\ell)+ 
\sum_{m=\ell+1}^k \Theta_m Z(t_{m-1},t_m) + \Theta_{k+1} Z(t_k,t_{\ell-1}+s+r),
\end{align*}
which can proven by splitting $Z(t_{j,\ell-1}-s)$ and $Z(r)$ in the left-hand side into the contributions due to the individual intervals, swapping the order of summation, and using the fact that 
$r$ lies in the interval $[t_{k,\ell-1}-s, t_{k+1,\ell-1}-s)$, in combination with the fact that all random variables on the right-hand side are independent due to the independent increments property of the L\'evy process $Z(\cdot)$. 

Finally, the third term describes the contribution due to the scenario that the client leaves after $t_d$:
\begin{align}\bar M^+_{{\bs t}, \ell,3}(\bs\omega, \bs\theta) &:=\int_0^{\delta_\ell}\frac{1}{\delta_\ell}
\int_{t_{d,\ell-1}-s}^\infty h(r)\,{\mathbb E}\exp\left(\sum_{j=\ell}^{d} \theta_j Z(t_{j,\ell-1}-s)
+\sum_{j=\ell}^{d} \omega_j
\right){\rm d}r\,{\rm d}s \nonumber \\
&=\int_0^{\delta_\ell}\frac{1}{\delta_\ell}
\int_{t_{d,\ell-1}-s}^\infty h(r)\,e^{\Omega_d-\Omega_{\ell-1}}\, (\varphi(\Theta_\ell))^{\delta_{\ell}-s}\cdot
\left(\prod_{m=\ell+1}^d (\varphi(\Theta_m))^{\delta_m}\right){\rm d}r\,{\rm d}s \label{Mbpl3}
.\end{align}

\iffalse 
\textcolor{black}{This may be incorrect
\begin{align*}
\bar{M}^+_{\bs t, \ell} (\bs \theta ) &= \int^{t_\ell}_{t_{\ell-1}} \frac{1}{\delta_\ell}\bigg(
\int_{s}^{t_{\ell}} h(r-s) \varphi(\Theta_\ell)^{r-s} {\rm d}r \\
&\quad +
\sum_{\ell'=\ell+1}^d \int_{t_{\ell'-1}}^{t_{\ell'}} h(r-s)\varphi( \Theta_\ell )^{t_\ell-s}\varphi(\Theta_{\ell'})^{r-t_{\ell'-1}}  \prod_{j=\ell+1}^{\ell'-1} \varphi(\Theta_{j})^{\delta_{j}} {\rm d}r \\
&\quad +\int_{t_d}^\infty h(r-s) \varphi( \Theta_\ell )^{t_\ell-s} \prod_{j=\ell+1}^d(\varphi(\Theta_{j}))^{\delta_j}{\rm d}r \bigg)  {\rm d}s
\end{align*}
}\fi

Combining the above, we have found the following expression for $M^+_{\bs t}(\bs\omega, \bs\theta)$.
Again it is fully in terms of the `partial sum series' $\Theta_k$ and $\Omega_k$.

\begin{lemma} \label{L2} For $ \bs\omega,\bs\theta\in{\mathbb R}^d$, we can compute $M^+_{\bs t}(\bs\omega, \bs\theta)$ by \eqref{Mp}, involving $\bar M^+_{{\bs t}, \ell}(\bs\omega, \bs\theta)$ via \eqref{Mpl}. Here $\bar M^+_{{\bs t}, \ell}(\bs\omega, \bs\theta)$  equals the sum of \eqref{Mbpl1}, \eqref{Mbpl2}, and \eqref{Mbpl3}.
\end{lemma}

With the moment generating functions given by the expressions above, we thus arrive at the large-deviation principle  presented in Proposition \ref{P2}.

%\begin{proposition}\label{P2}
%The vector  $(\bar G_n(t_1), \ldots, \bar G_n(t_d),
%\bar F_n(t_1),\ldots, \bar F_n(t_d))$ satisfies the LDP with rate function $I_{\bs t}({\bs g},{\bs f})$.
%\end{proposition}

\subsection{Sample-path LDP}\label{Sec:SPLDP}

We now establish the sample-path LDP of Theorem \ref{thm:main}. First observe that from the finite-dimensional LDP given in Proposition \ref{P2}, in combination with the Dawson--G\"artner projective limit theorem \cite[Thm.\  4.6.1]{DZ}, we obtain a sample-path LDP in the pointwise topology (which we denote by $\mathcal{X}$) with rate $n$ and rate function 
\begin{equation}\label{defI}
I^{\mathcal{X}}_{[0,T]}(f,g):= \sup_{d\in{\mathbb N}}\sup_{0\leqslant t_1< \dots < t_d \leqslant T} I_{(t_1, \dots, t_d)}((f(t_1), \dots, f(t_d)) ,(g(t_1), \dots, g(t_d))).
\end{equation}
{Recall that $I_{[0,T}(f,g)$ is the rate function characterised by \eqref{eq:IsD} and Lemmas \ref{L3} and \ref{L4} below.} To establish Theorem \ref{thm:main} we need to (i) show that  \[I^{\mathcal{X}}_{[0,T]}(f,g)=I_{[0,T]}(f,g),\] and (ii) strengthen the topology from $\mathcal{X}$ to the Skorokhod topology by establishing that {the sequence of bivariate processes $\{(\bar F_n(\cdot),\bar G_n(\cdot))\}_{n \in \mathbb{N}}$} is exponentially tight.

To establish (i) we need to verify that
\begin{itemize}
\item[(i-a)] $I^{\mathcal{X}}_{[0,T]}(f,g)=\infty$ when $f$ or $g$ is not absolutely continuous, 
\item[(i-b)] $I^{\mathcal{X}}_{[0,T]}(f,g) \leqslant I_{[0,T]}(f,g)$, and 
\item[(i-c)]  $I^{\mathcal{X}}_{[0,T]}(f,g) \geqslant I_{[0,T]}(f,g)$. 
\end{itemize}
In Section \ref{Sec:SPLDP1} we establish properties (i-b) and (i-c) {when the rate function is expressed in terms of limiting counterparts of the finite-dimensional mgf\,s $M^i(\bs\omega, \bs\theta)$, $i \in \{-,+\}$ that we derived in Section \ref{Sec:FDLDP}. We then find explicit expressions for these limiting mgf\,s.} 
In Section \ref{Sec:ExT} we prove properties (ii) and (i-a), exploiting the fact that they can be established via similar arguments. {This completes the proof the Theorem \ref{thm:main}.
Finally, in Section \ref{Sec:AF} we provide an alternate expression for $I _{[0,T]}(f,g)$ that may be attractive for computational purposes.}
 
%%Now that we have the LDP for finitely many points in time, the next step concerns a full sample-path LDP. As we will see below, this will involve an exponential tightness argumentation.
%In our sample-path LDP the role of the rate function of the path $(f,g)\equiv (g(s),f(s))_{s\in[0,T]}$ will be played by
%\begin{align*}
%I_{[0,T]}(f,g)&:=
%\sup_{\omega(\cdot),\theta(\cdot)} \bigg\{ \int^T_0 \left[ \theta(s)g(s) + \omega(s) f(s) \right]{\rm d}s\\ &\:\hspace{1cm}-\: f_0\log \mathbb{E} \exp \left( \int^T_0 \left[ \theta(s)G_-(s) + \omega(s)F_-(s) \right]{\rm d}s \right)  \\
%& \:\hspace{1cm}-\:\log \mathbb{E} \exp \left( \int_0^T \left[ \theta(s)G_+(s) + \omega(s)F_+(s) \right]{\rm d}s \right) \bigg\}.
%\end{align*}
%
%
%Importantly, the two mgf\,s that appear in this Legendre transform can be computed explicitly, as limits of the expressions for $M^-_{\bs t}(\bs \omega, \bs \theta)$ and $M^+_{\bs t}(\bs \omega, \bs \theta)$ that we derived in the previous section.  This we do in the first subsection, whereas in the second subsection we focus on the sample-path LDP.

\subsubsection{Upper and lower bounds {and the limiting mgf\,s}}\label{Sec:SPLDP1} In our construction we work with a mesh of dimension $d$ that we make increasingly fine. To this end, we define, for given functions $\omega(\cdot)$ and $\theta(\cdot)$,
\begin{equation}\label{scaleDefs}
d:=T/\Delta,\quad t_k:=k\Delta, \quad \theta_k:=\Delta\, \theta(k\Delta), \quad \omega_k=\Delta \,\omega(k \Delta), \quad \mbox{and} \quad  \delta_k:=\Delta.
\end{equation}
In addition, we introduce
\begin{equation}\label{eq:IsD}
I_{[0,T]}(f,g) =  \sup_{\omega(\cdot),\theta(\cdot)} \left( \int^T_0 \left[ \omega(s)\, f(s)+\theta(s)\,g(s)  \right]{\rm d}s - f_0 \log M_{[0,T]}^-(\omega, \theta) - \log M_{[0,T]}^+(\omega, \theta) \right),
\end{equation}
where the mgf\,s $M_{[0,T]}^i(\omega, \theta)$ are given by
\iffalse\footnote{Previously we had
\begin{align*}
\begin{split}
M_{[0,T]}^-(\omega, \theta)&:=
{\mathbb E}\exp\left(\int_0^T \left[\theta(s)G_-(s)+\omega(s)F_-(s) \right]\,{\rm d}s\right) = 
\lim_{\Delta\downarrow 0}  {\mathbb E}\exp\left(\sum_{k=1}^{T/\Delta}
\Delta \left[\theta(k\Delta)G_-(k\Delta)+\Delta \omega(k \Delta) F_-(k \Delta) \right] \right).
\end{split}
\end{align*}
but as far as I can see the second equality is essentially what we establish when we prove exponential tightness, so we shouldn't write this here.
} \fi
\begin{align}
\begin{split}\label{eq:MNLD}
M_{[0,T]}^i(\omega, \theta)&:= 
\lim_{\Delta\downarrow 0}  {\mathbb E}\exp\left(\sum_{k=1}^{T/\Delta}\Delta\,\omega(k \Delta) F_i(k \Delta)  +\sum_{k=1}^{T/\Delta}
\Delta \,\theta(k\Delta)G_i(k\Delta)\right), \quad i \in \{-,+\};
\end{split}
\end{align}
the supremum in \eqref{eq:IsD} is taken over all continuous bounded functions on $[0,T]$.
\begin{lemma}
If $f$ and $g$ are absolutely continuous, then $I^{\mathcal{X}}_{[0,T]}(f,g) \geqslant I _{[0,T]}(f,g)$.
\end{lemma}
\begin{proof}
We have
\begin{align}
\begin{split}\label{eq:IsD2}
    I^{\mathcal{X}}_{[0,T]}(f,g) &\geqslant \lim_{\Delta \downarrow 0} I_{(t_1, \dots, t_d)}\big((f(t_1), \dots, f(t_d)) ,(g(t_1), \dots, g(t_d))\big) \\
    &=\lim_{\Delta \downarrow 0}\sup_{\omega(\cdot),\theta(\cdot)} \left(\sum_{j=1}^d \omega_j f(t_j) +\sum_{j=1}^d \theta_j g(t_j)- f_0\log M^-_{\bs t}(\bs\omega, \bs\theta) - \log M^+_{\bs t}(\bs\omega, \bs\theta)\right) \\
    &\geqslant I _{[0,T]}(f,g),
    \end{split}
\end{align}
where the first inequality is due to the definition \eqref{defI}, and 
where the second inequality follows by {contradiction. Indeed, suppose that the inequality does not hold. Then there exists $\omega^\star , \theta^\star $ such that 
\begin{align*}
      \int^T_0[\omega^\star (s)\,f( & s)+\theta^\star (s)\,g(s)]{\rm d}s - f_0 \log M^-_{[0,T]}(\omega^\star ,\theta^\star )-\log M^+_{[0,T]}(\omega^\star , \theta^\star ) \\
    &< \lim_{\Delta \downarrow 0} \sup_{\omega(\cdot), \theta(\cdot)} \left( \sum_{j=1}^d \omega_j \,f(t_j) + \sum^d_{j=1} \theta_j \,g(t_j) - f_0\log M^-_{\bs{t}}(\bs{\omega}, \bs{\theta}) - \log M^+_{\bs{t}}(\bs{\omega}, \bs{\theta})  \right);
\end{align*}
however, if we ignore the supremum on the right-hand side and replace $(\omega,\theta)$ by $(\omega^\star , \theta^\star )$ we obtain equality, i.e., a contradiction.}
% by recognising that there is equality between the right-hand sides of \eqref{eq:IsD} and \eqref{eq:IsD2} when the supremum is ignored and we instead consider specific functions $\omega(\cdot)$ and $\theta(\cdot)$.
%\footnote{MM: I recall there was a moment that I understood this argument, but I fail to see how this works now. Why the second inequality?}
\end{proof}

\begin{lemma} 
If $f$ and $g$ are absolutely continuous then $I^{\mathcal{X}}_{[0,T]}(f,g) \leqslant I _{[0,T]}(f,g)$.
\end{lemma}
\begin{proof}
Suppose to the contrary that  
$I^{\mathcal{X}}_{[0,T]}(f,g) > I _{[0,T]}(f,g)$ for some absolutely continuous $f$ and $g$.
In that case there must exist a vector $\bs t = (t_1, \dots, t_d)$ such that 
\begin{equation}\label{eq:c1}
I_{\bs t}\big((f(t_1), \dots, f(t_d)),(g(t_1), \dots, g(t_d)) \big)>I_{[0,T]}(f,g).
\end{equation}
For $\ell \in \mathbb{N}$, let $\bs s^{\ell}=(s^{[\ell]}_1, \dots , s^{[\ell]}_{k_\ell})$ be such that 
\begin{itemize}\item[$\circ$] for any $i\in\{1,\ldots,d\}$ there exists $ j\in\{1,\ldots, k_{\ell}\}$ with $t_i=s^{[\ell]}_j$,
\item[$\circ$] $\lim_{\ell \to \infty} \max_{j\in\{1 ,\ldots, k_\ell\}} |s^{[\ell]}_j -s^{[\ell]}_{j-1}|=0$. 
\end{itemize}
By the contraction principle, for any $\ell \geqslant 1$, we have
\begin{equation}\label{eq:c2}
I_{\bs t}((f(t_1), \dots, f(t_d),(g(t_1), \dots, g(t_d)))\leqslant I_{\bs s^{[\ell]}}\left((f(s^{[\ell]}_1), \dots, f(s^{[\ell]}_{k_\ell})),(g(s^{[\ell]}_1), \dots, g(s^{[\ell]}_{k_\ell}))\right).
\end{equation}
If we can show that 
\begin{equation}\label{eq:CTP}
\lim_{\ell \to \infty} I_{\bs s^{[\ell]}}\left((f(s^{[\ell]}_1), \dots, f(s^{[\ell]}_k)),(g(s^{[\ell]}_0), \dots, g(s^{[\ell]}_k))\right) = I _{[0,T]}(f,g),
\end{equation}
then, when combined with \eqref{eq:c2}, we have contradicted \eqref{eq:c1} and hence proved the result. 
To establish \eqref{eq:CTP} it suffices that we verify that arguments in the proofs of Lemmas \ref{L3} and \ref{L4}  below still apply when \[\theta^{[\ell]}_k:=(s^{[\ell]}_{k+1}-s^{[\ell]}_k)\, \theta(s^{[\ell]}_k),\:\:\:\omega^{[\ell]}_k:=(s^{[\ell]}_{k+1}-s^{[\ell]}_k)\, \omega(s^{[\ell]}_k)\] and $\ell \to \infty$ (rather than {$\theta_k=\Delta \theta(k \Delta)$ and $\omega_k=\Delta \omega(k \Delta)$ and} $\Delta \to 0$). As this verification is of a rather mechanical nature, we do not include it here.
% \footnote{PB: This is sketchy. Any improvements are welcome. MM: argument looks OK to me, but perhaps we should expand on it?}
\end{proof}

%where $M^-_{[0,T]}(\omega, \theta)$ and $M^+_{[0,T]}(\omega, \theta)$ are the limiting moment generating functions defined in \eqref{eq:MNLD} and \eqref{eq:MPLD} respectively.
%In Section \ref{Sec:SPLDP1} we compute these limiting moment generating functions and, in Section \ref{Sec:AF}, we show that the resulting expression for $I^\star _{[0,T]}(f,g)$ does indeed correspond to $I_{[0,T]}(f,g)$ as defined prior to Theorem \ref{thm:main}. In combination this implies (i-c).

Our next task is to compute the limiting mgf\,s $M_{[0,T]}^-(\omega, \theta)$ and $M_{[0,T]}^+(\omega, \theta)$.
We start with the (somewhat easier) first mgf, i.e., the one pertaining to $G_-(\cdot)$ and $F_-(\cdot)$.  Let \[\Theta(s) =\int_s^T \theta(r)\,{\rm d}r,\:\:\:\: \Omega(s):=\int_0^s \omega(r)\,{\rm d}r,\]
i.e., the counterparts of the objects $\Theta_k$ and $\Omega_k$ that we worked with in the finite-dimensional context,
    and
    \[\Psi_{\omega, \theta}(u)=\Omega(u)+ \int_0^u\log \varphi\left( 
 \Theta(s)
\right) {\rm d}s.\]
\begin{lemma}\label{L3} For  $\omega\equiv \omega(\cdot)$ and $\theta\equiv \theta(\cdot)$,
\[M_{[0,T]}^-(\omega, \theta) = \int_0^T h^\circ (u) 
\,e^{\Psi_{\omega, \theta}(u)} {\rm d}u + \bar h^\circ (T) 
\,e^{\Psi_{\omega, \theta}(T)} .\]
\end{lemma}
\begin{proof}
%To this end, we first observe that, using a Riemann sum representation,
%\begin{align}
%\begin{split}\label{eq:MNLD}
%M_{[0,T]}^-(\omega, \theta)&:=
%{\mathbb E}\exp\left(\int_0^T \left[\theta(s)G_-(s)+\omega(s)F_-(s) \right]\,{\rm d}s\right) \\&= 
%\lim_{\Delta\downarrow 0}  {\mathbb E}\exp\left(\sum_{k=1}^{T/\Delta}
%\Delta \left[\theta(k\Delta)G_-(k\Delta)+\Delta \omega(k \Delta) F_-(k \Delta) \right] \right).
%\end{split}
%\end{align}
%To evaluate the right-hand-side we use the expression 
%\eqref{Mme}, as identified in Lemma \ref{L1}, with 
%\begin{equation}\label{scaleDefs}
%d:=T/\Delta,\quad t_k:=k\Delta, \quad \theta_k:=\Delta \theta(k\Delta), \quad \omega_k=\Delta \omega(k \Delta), \quad \mbox{and} \quad  \delta_k:=\Delta,
%\end{equation}
%where we let $T/\Delta$ be an integer.
% As a next step we compute 
% the finite-dimensional mgf in the right-hand side; that is, we use the  expressions found above, and pick  with $d:=T/\Delta$, $t_k:=k\Delta$, $\theta_k:=\Delta\,\theta(k\Delta)$, and $\delta_k:=\Delta.$ 
Concerning the second term in the right-hand side of \eqref{Mme},  recognising Riemann sums, we readily obtain
\begin{align*}\bar h^\circ(T)\,
\lim_{\Delta\downarrow 0} \exp &\left(\sum^{T/\Delta}_{k=1} \Delta \omega(k\Delta) \right)\cdot \prod_{k=1}^{T/\Delta}\left(\varphi\left(\sum_{j=k}^{T/\Delta}\Delta\,\theta(j\Delta)\right)\right)^\Delta \\
&=\bar h^\circ(T)\,\lim_{\Delta\downarrow 0} \exp\left(
\sum_{k=1}^{T/\Delta}\Delta \left[\omega(k \Delta) + \log \varphi\left(\sum_{j=k}^{T/\Delta}\Delta\,\theta(j\Delta)\right)
 \right]\right)\\
&=\bar h^\circ(T)\,\exp\left(
\int_0^T \left[ \omega(s)+ \log \varphi\left( 
\int_s^T \theta(r)\,{\rm d}r
\right) \right]{\rm d}s\right)=\bar h^\circ(T)\,e^{\Psi_{\omega, \theta}(T)}.
    \end{align*}
%    where $\Theta(s) =\int_s^T \theta(r)\,{\rm d}r$ and $\Omega(s):=\int_0^s \omega(r)\,{\rm d}r$
%    and
%    \[\Psi_{\omega, \theta}(u):=\Omega(u)+ \int_0^u\log \varphi\left( 
% \Theta(s)
%\right) {\rm d}s.\]
    We continue by focusing on the first term in the right-hand side of \eqref{Mme}. We find, again recognising various Riemann sums, 
    \begin{align*}
&\lim_{\Delta\downarrow 0} {\mathbb E}\exp \left(\sum_{k=1}^{T/\Delta}
\Delta \left[ \omega(k\Delta)F_-(k\Delta)+\theta(k\Delta)G_-(k\Delta)  \right]\right) \\
&= \lim_{\Delta\downarrow 0} \sum_{k=1}^{T/\Delta} 
\int_{(k-1)\Delta}^{k\Delta} 
h^\circ (s) \varphi\left( \sum_{\ell=1}^{k-1}\Delta \theta(\ell \Delta) \right)^{s-(k-1)\Delta}\hspace{-3mm} {\rm d}s\,\exp \left( \sum_{\ell=1}^{k-1}\Delta \omega(\ell \Delta) \right) \prod_{\ell=1}^{k-1}\left(\varphi\left(\sum_{j=\ell}^{T/\Delta}\Delta\,\theta(j\Delta)\right)\right)^\Delta\\
&= \lim_{\Delta\downarrow 0} \sum_{k=1}^{T/\Delta} 
{\Delta} 
h^\circ (k\Delta)
%(1+o(1))\, 
\exp \left( \sum_{\ell=1}^{k-1}\Delta \left[ \omega(\ell \Delta) + \log \varphi\left(\sum_{j=\ell}^{T/\Delta}\Delta\,\theta(j\Delta)\right) \right] \right) \\
&= \int_0^T h^\circ (u) 
\,\exp\left(
\int_0^u \left[\omega(s)+ \log \varphi\left( 
\int_s^T \theta(r)\,{\rm d}r
\right) \right]{\rm d}s\right) {\rm d}u=\int_0^T h^\circ (u) 
\,e^{\Psi_{\omega, \theta}(u)} {\rm d}u.
\end{align*}
This completes the proof.
\end{proof}

The next step is to evaluate the mgf $M_{[0,T]}^+(\omega, \theta)$. We show that it can be expressed in terms of the objects
\[
\Phi_{\omega, \theta}(s):= \int_s^T  h(r-s)\,e^{\Psi_{\omega, \theta}(r)-
\Psi_{\omega, \theta}(s) }{\rm d}r\qquad \text{and} \qquad \bar\Phi_{\omega, \theta}(s):=\bar h(T-s)\,e^{\Psi_{\omega, \theta}(T)-
\Psi_{\omega, \theta}(s)}.
\]
\begin{lemma}
\label{L4} For $\theta\equiv \theta(\cdot)$ and $\omega\equiv \omega(\cdot)$,
\[M_{[0,T]}^+(\omega, \theta) = \exp\left(\lambda
\int_0^T \big(
\Phi_{\omega, \theta}(s)+\bar\Phi_{\omega, \theta}(s)-1
\big)\,{\rm d}s
\right).\]
\end{lemma}
\begin{proof}
%\begin{align}
%\begin{split}\label{eq:MPLD}
%M_{[0,T]}^+(\omega, \theta)&:=
%{\mathbb E}\exp\left(\int_0^T \left[\theta(s)G_+(s)+\omega(s)F_+(s) \right]\,{\rm d}s\right) \\&= 
%\lim_{\Delta\downarrow 0}  {\mathbb E}\exp\left(\sum_{k=1}^{T/\Delta}
%\Delta \left[\theta(k\Delta)G_+(k\Delta)+\Delta \omega(k \Delta) F_+(k \Delta) \right] \right).
%\end{split}
%\end{align}
We proceed in a similar manner as above (again working with the various quantities that were defined in \eqref{scaleDefs}). In this case, however, the expression for $M^+_{\bs t}(\bs\omega, \bs\theta)$, as was provided in Lemma~\ref{L2}, is considerably more complex. We first analyze the quantities \eqref{Mbpl1}, \eqref{Mbpl2}, and \eqref{Mbpl3} (under the parametrization given in~\eqref{scaleDefs})  \textcolor{black}{when} $\Delta$ is small. The results allow us to compute $M_{[0,T]}^+(\omega, \theta)$ using $\eqref{Mp}$ and \eqref{Mpl}. The starting point is that, with ${\bs t}$, ${\bs \theta}$ and ${\bs \omega}$ as in~\eqref{scaleDefs},
\[\bar M^+_{{\bs t}, \ell}(\bs\omega, \bs\theta) = \exp\left(\lambda\Delta \sum_{\ell=1}^{T/\Delta} 
\left(\bar M_{{\bs t},\ell}^+(\bs\omega, \bs\theta)-1
\right)
\right)
=
 \exp\left(\lambda\Delta \sum_{\ell=1}^{T/\Delta} 
\bar M_{{\bs t},\ell}^+(\bs\omega, \bs\theta)-\lambda T
\right)
.\]
Then recall that $\bar M_{{\bs t},\ell}^+(\bs\omega, \bs\theta)$ is the sum of \eqref{Mbpl1}, \eqref{Mbpl2}, and \eqref{Mbpl3}.

We start by considering the contribution due to \eqref{Mbpl1}. 
Recall that this corresponds to the case where a client arrives and leaves in the same time interval. When this time interval, $\Delta$, is becoming infinitely small, it is expected that this term does not play any role in the arguments to come. Indeed, it is elementary to show that 
\[\lambda\Delta \sum_{\ell=1}^{T/\Delta}
\bar M_{{\bs t},\ell,1}^+(\bs\omega, \bs\theta)=
O(\Delta)\] as $\Delta\downarrow 0$, 
which justifies leaving it out in the rest of the derivation.

Then focus on the contribution due to \eqref{Mbpl2}. This corresponds to the clients who have arrived in the time interval $[t_{\ell-1},t_\ell)$ and then leave before $T$. We note that the $s$ in \eqref{Mbpl2} lies between $0$ and $\Delta$, so that it can be argued  that \textcolor{black}{when} $\Delta$ gets small, we can replace it by $0$. This concretely means, with $d=T/\Delta$,
\begin{align*}\lim_{\Delta\downarrow 0}
\lambda\Delta &\sum_{\ell=1}^{T/\Delta}
\bar M_{{\bs t},\ell,2}^+(\bs\omega, \bs\theta)\\
&=
\lim_{\Delta\downarrow 0}
\lambda\Delta \sum_{\ell=1}^{T/\Delta}
\sum_{k=\ell}^{d-1}\int_{(k-\ell)\Delta}^{(k-\ell+1)\Delta} h(r)
\,e^{\Omega(Tk/d)-\Omega(T\ell/d)}\cdot
\left(\prod_{m=\ell}^k (\varphi(\Theta(Tm/d)))^{\Delta}\right){\rm d}r\\
&=
\lim_{\Delta\downarrow 0}
\lambda\Delta^2 \sum_{\ell=1}^{T/\Delta}
\sum_{k=\ell}^{T/\Delta}h((k-\ell)\Delta)
\,e^{\Omega(k\Delta)-\Omega(\ell\Delta)}\exp
\left(\Delta\sum_{m=\ell}^k \log \varphi(\Theta(m\Delta))\right)
.\end{align*}
Recognising various Riemann sums, we thus obtain
\begin{align*}\lim_{\Delta\downarrow 0}&
 \exp\left(\lambda\Delta \sum_{\ell=1}^{T/\Delta} 
\bar M_{{\bs t},\ell,2}^+(\bs\omega, \bs\theta)
\right) \\
&= \exp\left(
\lambda \int_0^T\int_s^T  h(r-s)\,e^{\Omega(r)-\Omega(s) }
\exp\left(\int_s^r \log\varphi(\Theta(u)){\rm d}u\right){\rm d}r\,{\rm d}s
\right)=\exp\left(\lambda \int_0^T \Phi_{\omega, \theta}(s)\,{\rm d}s\right),
\end{align*}
with
$\Phi_{\omega, \theta}(s):= \int_s^T  h(r-s)\,e^{\Psi_{\omega, \theta}(r)-
\Psi_{\omega, \theta}(s) }{\rm d}r.$

\iffalse
Note that below $s$ and $u$ can be interpreted as the times that the client enters and leaves the system respectively. We have
\begin{align}
\eqref{Mbpl2} &= \int_{s}^{s+\Delta}\frac{{\rm d} v }{\Delta}
\sum_{u/\Delta=s/\Delta}^{T/ \Delta-1}\int_{u}^{u+\Delta}{\rm d} r h(r-v)\,e^{\Omega_{s/\Delta}-\Omega_{u/\Delta+1}}\, (\varphi(\Theta_{s/\Delta}))^{s+\Delta-v} \nonumber \\
&\qquad\cdot
\left(\prod_{m=s /\Delta+1}^{u/\Delta} (\varphi(\Theta_m))^{\Delta}\right)
\cdot((\varphi(\Theta_{u/\Delta-1}))^{r-u}. \nonumber \\
&= \sum_{u/\Delta=s/\Delta}^{T/\Delta-1} \Delta h(u-s)e^{\sum_{m=s/\Delta}^{u/\Delta} \Delta \omega(m \Delta)}\cdot \exp\left\{ \sum_{v=s/\Delta+1}^{u/\Delta} \Delta \log \left( \sum_{r=v}^{u/\Delta} \Delta \theta(r \Delta) \right)\right\}(1+O(\Delta)) \label{T11} \\
&= \sum^{T/\Delta}_{u/\Delta=s/\Delta} h(u-s) \exp \left\{ \int_{s}^u {\rm d} r \left[ \omega(r) + \log \left( \int_v^u \theta(v) {\rm d}v \right) \right] \right\}(1+o(1) \label{T12} \\
 &= \int^T_s {\rm d} u h(u-s) \exp \left\{ \int_s^u {\rm d}r \left[ \omega(r) + \log \left( \int_r^u \theta(v) {\rm d}v \right)  \right] \right\}(1+o(1)), \label{Ml2}
\end{align} 
where in \eqref{T11} we use the fact that $\varphi(\cdot)^\Delta=1+O(\Delta)$ and in \eqref{T12} and \eqref{Ml2} we recognise various Riemann sums. 
\fi

We conclude by analyzing the contribution due to \eqref{Mbpl3}, describing the impact of the clients who arrive in $[t_{\ell-1},t_{\ell})$ and remain in the system until time $T$.
Just as we did for the contribution due to \eqref{Mbpl2}, observe that the $s$ in \eqref{Mbpl3} lies between $0$ and $\Delta$; it again requires a standard argument to justify that \textcolor{black}{when} $\Delta\downarrow 0$ we can replace it by $0$. More concretely, with $d=T/\Delta$,
\[\lim_{\Delta\downarrow 0}
\lambda\Delta \sum_{\ell=1}^{T/\Delta}
\bar M_{{\bs t},\ell,3}^+(\bs\omega, \bs\theta)=
\lim_{\Delta\downarrow 0}
\lambda\Delta \sum_{\ell=1}^{T/\Delta}\int_{(d-\ell)\Delta}^\infty h(r)
\,e^{\Omega(T)-\Omega(T\ell/d)}\cdot
\left(\prod_{m=\ell+1}^d (\varphi(\Theta(Tm/d)))^{\Delta}\right){\rm d}r,
\]
which can be rewritten as
\[\lim_{\Delta\downarrow 0}
\lambda\Delta \sum_{\ell=1}^{T/\Delta}
\bar h(T-\ell\Delta) 
\,e^{\Omega(T)-\Omega(\ell\Delta)}\exp
\left(\Delta\sum_{m=\ell+1}^{T/\Delta} \log \varphi(\Theta(m\Delta))\right).\]
We thus conclude that
\begin{align*}\lim_{\Delta\downarrow 0}
 \exp\left(\lambda\Delta \sum_{\ell=1}^{T/\Delta} 
\bar M_{{\bs t},\ell,3}^+(\bs\omega, \bs\theta)
\right) &= \exp\left(
\lambda \int_0^T \bar h(T-s)\,e^{\Omega(T)-\Omega(s) }
\exp\left(\int_s^T \log\varphi(\Theta(r)){\rm d}r\right){\rm d}s
\right)\\
&=\exp\left(\lambda \int_0^T \bar\Phi_{\omega, \theta}(s)
{\rm d}s\right),
\end{align*}
with
$\bar\Phi_{\omega, \theta}(s)=\bar h(T-s)\,e^{\Psi_{\omega, \theta}(T)-
\Psi_{\omega, \theta}(s)}.$
\end{proof}

\subsubsection{Exponential tightness}\label{Sec:ExT}

%\textit{*This section still needs some work but I think it is currently in a readable state.}

To establish exponential tightness we rely on the approach that was developed in Feng and Kurtz \cite{FK}. With $X_n(t):=(F_n(t),G_n(t))$,
we first need a metric $r$ on $\mathbb{R}^2$. To this end, for $x=(x_1,x_2)$ and $y=(y_1,y_2)$ define
\[
r(x,y):=|x_1 - y_1| + |x_2-y_2|,
\]
and we let $q(x,y):= r(x,y)\wedge 1$. Let $D([0,\infty))$ be the c\`adl\`ag space in which the trajectories of $X_n(\cdot)$ are contained {and equip it with the Skorokhod topology}. In the sequel $\{\mathscr{F}^n_t\}_{0 \leqslant t \leqslant T}$ is a (naturally chosen) filtration that we  detail below.
{In this case} \cite[Theorem 4.1]{FK} implies the following:
\begin{quotation}
%Let $\mathcal{T}_0$ be a dense subset of $[0,\infty)$. 
\emph{Suppose that
\begin{itemize}
\item[(A)]
$\{X_n(t) \}_{n \in {\mathbb N}}$ is exponentially tight for each $t \geqslant 0$ and
\item[(B)]
for each $T>0$, there exists random variables $\gamma_n(\delta, \alpha, T)$, satisfying
    \begin{equation}\label{tightnessC1}
    \mathbb{E} \left[ \left. e^{n\alpha q(X_n(t+u),X_n(t))} \right| {\mathscr F}_t^n \right] \leqslant \mathbb{E}\left[ \left. e^{\gamma_n(\delta, \alpha, T)} \right| {\mathscr F}_t^n  \right]
    \end{equation}
    for $0 \leqslant t \leqslant T$ and $0 \leqslant u \leqslant \delta$ such that for each $\alpha>0$,
    \begin{equation}\label{tightnessC2}
    \lim_{\delta \downarrow 0} \limsup_{n\to \infty} \frac{1}{n} \log \mathbb{E} \left[ e^{\gamma_n(\delta, \alpha,T)} \right]=0.
    \end{equation}
    % and 
    % \begin{equation}\label{tightnessC3}
    % \lim_{\delta \downarrow  0} \limsup_{n \to \infty} \frac{1}{n} \log \mathbb{E} \left[ e^{n \lambda q^{\beta(X_n(\delta,X_n(0)}} \right]=0. 
    % \end{equation}
    \end{itemize}
 Then $\{X_n(\cdot)\}_{n\geqslant 0}$ is exponentially tight in $D[0,\infty)$.}
 \end{quotation}
%Theorem 4.1 of Feng and Kurtz also contains two other equivalent conditions (c) and (d). Condition (c) is given at the end of this section.
%
%We first establish exponential tightness under a simplifying assumption. Recall that $h(\cdot)$ denotes the density of the service time distribution. We suppose that 
%\begin{equation}\label{Tightnessassumption}
%h(s) \bigg/ \int^\infty_s h(s) < C,
%\end{equation}
%for some constant $C$ that does not depend on $s$. However, in this document we will try to avoid this assumption.
%We first do this under the additional assumption that there are no clients initially in the queue. 

% We suppose that the residual sojourn times of the clients that are initially in the system have a bounded density $h^\circ(\cdot)$, i.e., there exists $C<\infty$ such that $h^\circ(t) \leqslant C$ for all $t \geqslant 0$.
Observe that, by Proposition \ref{P1} (i.e., the LDP pertaining to a single point in time), it follows that for each given value of $t \geqslant 0$ the sequence $\{ X_n(t)\}_{n \in{\mathbb N}}$ is {exponentially} tight, so that the  requirement (A) has been taken care of. Hence, to prove Theorem \ref{thm:main}, we have to verify requirement (B), i.e., Condition \eqref{tightnessC1} and Condition \eqref{tightnessC2}.

Before we verify Condition \eqref{tightnessC1} and Condition \eqref{tightnessC2}, we first discuss 
% \footnote{MM: I left out `informally', is this would lead w.p.\ 1 to the event that the reviewer asks us to make it formal, which doesn't necessarily make the paper stronger and more accessible.} 
the filtration $\{{\mathscr F}_t^n\}_{0 \leqslant t \leqslant T}$. In view of the proofs to follow, we do so by  describing the information that is contained in ${\mathscr F}_t^n$.
Given ${\mathscr F}_t^n$ we in the first place know the time that each client arrives to the system up to time~$t$; we label these times as $\tau_1< \tau_2< \dots< \tau_{A_n(t)}$, with $A_n(t)$ the number of client arrivals until time~$t$. 
The $i$-th arrival is assigned a sojourn time $S_i$ and, given ${\mathscr F}_t^n$, we know whether $S_i \leqslant t-\tau_i$ and if this inequality holds then we know the precise value of $S_i$. 
Given ${\mathscr F}_t^n$ we in addition know whether or not each individual initially present has left the system and, if so, we know her specific residual sojourn time; the residual time of the $i$-th client is represented by $S^\circ_i$.
We also know the claim sizes and claim arrival times pertaining to all the claims that occurred in the time interval $(0,t]$.

{\it $\circ$~Step I: Construction of $\gamma_n(\delta, \alpha,T)$ so that Condition \eqref{tightnessC1} is met.} 
To apply \cite[Theorem 4.1]{FK}, the idea is to identify a random variable $\gamma_n(\delta, \alpha,T)$ that stochastically dominates 
\begin{align*}
n \alpha\, q\big((F_n(t+u), G_n(t+u)),\,&(F_n(t), G_n(t))\big)\\&\leqslant n \alpha\, (| F_n(t+u)- F_n(t)| \wedge 1) + n \alpha\, (| G_n(t+u)-G_n(t)| \wedge 1)
\end{align*}
for any $\alpha \geqslant 0$, $0 \leqslant u \leqslant \delta$, $0\leqslant t \leqslant T$, and ${\mathscr F}^n_t$.
To this end, we first find a stochastically dominating random variable for $| F_n(t+u)- F_n(t)| \wedge 1$. 
Recall that $\{ A_n(t): 0 \leqslant t \leqslant T \}$ denotes the arrival process of the clients, i.e.,  a Poisson process with rate $\lambda n$.
Observe that the change in the number of clients in the system between times $t$ and $t+u$ (with $u\in[0,\delta]$) is dominated by the number of clients who arrived in $(t,t+\delta)$ plus the number of clients that were served in this time interval. 
Now let $\bar A_n(\delta) \sim {\rm {\mathbb P}oi}(n \lambda \delta)$, and let this random quantity be independent of the client arrival process $A_n(\cdot)$. Note that, because $u \in[0,\delta]$, $\bar A_n(\delta)$ stochastically dominates the number of clients who arrive in the interval $(t,t+u)$ given ${\mathscr F}^n_t$. 
%Now note that, if $B_1, B_2, \dots$ is a sequence of i.i.d. Bernoulli random variables with probability $\delta C$ (independent of everything else) then 
%\begin{equation}\label{Gme1}
%\sum_{i=1}^{nf_0} B_i+\int_0^t {\rm d} u \sum_{i=1}^{A(t)} \boldsymbol 1 \{ \tau_i=u\} \boldsymbol  1 \{ S_i \in [t-u,t-u+\delta] \},
%\end{equation}
%stochastically dominates the number of customers served in $(t,t+u)$. To better understand \eqref{Gme1} observe that, given ${\mathscr F}_t^n$, the first term in \eqref{Gme1} stochastically dominates the number of customers served in $(t, t+\delta)$ who were initially present in the queue\footnote{there is an issue here: this doesn't hold for distributions with bounded support for example.}, and the second term in \eqref{Gme1} is equal to the number of customers served in $(t, t+ \delta)$ who were {not} initially present in the queue.
In addition, note that the number of clients who leave the system between times $t$ and $t+u$ is dominated by the number that left in $(t,t+\delta)$, which is given by 
\begin{equation}\label{Gme1}
V_n(\delta,t):=\sum_{i=1}^{nf_0} \boldsymbol 1\{ S_i^\circ \in [t,t+\delta]\}+\int_0^t {\rm d} u \sum_{i=1}^{A(t)} \boldsymbol 1 \{ \tau_i=u\} \boldsymbol  1 \{ S_i \in [t-u,t-u+\delta] \}.
\end{equation}
Note that \eqref{Gme1} is a function of $t$, whereas the dominating random variable that we must construct $\gamma_n(\delta, \alpha, T)$ should not depend on $t$. We thus dominate $\eqref{Gme1}$ by 
$
\sup_{t \in [0,T]}V_n(\delta,t)$
to obtain a bound that is uniform in $t\in[0,T]$.
Taking into account both arrivals and departures, we then have, for any ${\mathscr F}_t^n$, 
\begin{align*}
| F_n(t+u)- F_n(t)| \wedge 1  &\stackrel{\rm st}{\leqslant} \frac{\bar A_n(\delta)+\sup_{t \in [0,T]}V_n(\delta,t)}{n}\wedge 1 \\
 &\stackrel{\rm st}{\leqslant}
\beta^{(1,1)}_n(\delta, T)+ \beta^{(1,2)}_n(\delta, T)+ \beta^{(1,3)}_n(\delta, T) =: \beta^{(1)}_n(\delta, T),
\end{align*}
where we define
\begin{align*}
\beta^{(1,1)}_n(\delta, T)&:= \frac{\bar A_n(\delta)}{n},\:\:\:\:\:
\beta^{(1,2)}_n(\delta, T):=\frac{1}{n}
\sup_{t \in [0,T]}\left\{ \sum_{i=1}^{nf_0} \boldsymbol 1\{ S_i^\circ \in [t,t+\delta]\} \right\}, \\
\beta^{(1,3)}_n(\delta, T)&:= \frac{1}{n}\left( \sup_{t \in [0,T]} \left\{\int_0^t {\rm d} u \sum_{i=1}^{A(T)} \boldsymbol 1 \{ \tau_i=u\} \boldsymbol  1 \{ S_i \in [t-u,t-u+\delta] \}\right\} \wedge n \right).
\end{align*}

Now that we have succeeded in identifying a stochastically dominating random variable for the first component $| F_n(t+u)- F_n(t)| \wedge 1$, 
we proceed by identifying a stochastically dominating random variable for the second component $| G_n(t+u)- G_n(t)| \wedge 1$. The change in the net aggregate claim process between times $t$ and $t+u$ is dominated by the premiums paid by the clients in this time interval plus the claims made by the clients in this time interval. 
Recalling that $u\in[0,\delta]$, the premiums paid by the clients between times $t$ and $t+u$ is dominated by $r\eta_n(\delta,T)$ with $\eta_n(\delta,T):= \delta (n f_0 + A_n(T))$, and the sum of the claims made by the clients between times $t$ and $t+u$ is dominated by
\begin{equation}
    \label{eq:defy}
\bar Y_n(\delta, T):=
\sum_{i=1}^{\bar A_n(\delta,T)} Y_i, 
\end{equation}
where $\bar A_n(\delta,T) \sim {\rm {\mathbb P}oi}(\nu \eta_n(\delta,T))$ and is conditionally independent of everything else given $A_n(T)$, and the $Y_i$ are iid random variables with mgf $\beta(\cdot)$. 
Hence, for any ${\mathscr F}_t^n$,
\[
| G_n(t+u)- G_n(t)| \wedge 1 \stackrel{\rm st}{\leqslant}   \frac{1}{n}\left(r \delta (n f_0 + A_n(T))+\bar Y_n(\delta, T)\right)\wedge 1 =: \beta^{(2)}_n(\delta, T).
\]
From the above we conclude that Condition \eqref{tightnessC1} is satisfied if we let 
\[
\gamma_n(\delta,\alpha,T) := \alpha n \left( \beta^{(1)}_n(\delta, T)+\beta^{(2)}_n(\delta, T)\right).
\]
We conclude that we have constructed a random quantity  $\gamma_n(\delta,\alpha,T)$ so that Condition \eqref{tightnessC1} is met.
% \begin{align*}
% \gamma_n(\delta,\alpha,T) &= \alpha\left[\left(\bar A_n(\delta)+ \sum_{i=1}^{nf_0} B_i +\sup_{t \in [0,T]} \left\{ \int_0^t {\rm d} u \sum_{i=1}^{A(t)} \boldsymbol 1 \{ \tau_i=u\} \boldsymbol  1 \{ S_i \in [t-u,t-u+\delta] \} \right\} \right)\wedge n \right] \\
% &\qquad + \alpha\left[ \left(r \delta (n f_0 + A_n(T))+\sum_{i=1}^{TC} Y_i \right)\wedge n \right]
% \end{align*}

{\it $\circ$~Step II: Verifying that $\gamma_n(\delta, \alpha,T)$ is so that Condition \eqref{tightnessC2} is met.} 
%\footnote{\small PB: Maybe this could be a lemma?} 
Now we need to verify Condition \eqref{tightnessC2}, i.e., we need to show that, for the constructed $\gamma_n(\delta,\alpha,T)$, and for any $\alpha > 0$,
\[
\lim_{\delta \downarrow 0} \limsup_{n \to \infty} \frac{1}{n} \log \mathbb{E} \left[ e^{\gamma_n(\delta, \alpha,T)} \right]=0.
\]
To this end, first observe that by H\"{o}lder's inequality
\[
\frac{1}{n} \log \mathbb{E} \left[ e^{\gamma_n(\delta, \alpha,T)} \right] \leqslant \frac{1}{2n} \log \mathbb{E} \left[ e^{2\alpha\, n \beta^{(1)}_n(\delta, T)} \right]+ \frac{1}{2n} \log \mathbb{E} \left[ e^{2\alpha \,n \beta^{(2)}_n(\delta, T)} \right].
\]
Hence to verify \eqref{tightnessC2} we can separately treat each term in the right-hand side of the previous display. 
%\footnote{PB: Maybe this could be a lemma?} 
We start by establishing 
\begin{equation}\label{CB1}
\lim_{\delta \downarrow 0} \limsup_{n \to \infty} \frac{1}{n} \log \mathbb{E} \left[ e^{2 \alpha n \beta^{(1)}_n(\delta, T)} \right]=0.
\end{equation}
Because $\beta_n^{(1,1)}(\delta, T)$, $\beta_n^{(1,2)}(\delta, T)$, and $\beta_n^{(1,3)}(\delta, T)$ are independent, \eqref{CB1} follows from the following lemma.
\begin{lemma} \label{LEM1} For $i=1,2,3$ it holds that
\begin{equation}\label{CB11}
\lim_{\delta \downarrow 0} \limsup_{n \to \infty} \frac{1}{n} \log \mathbb{E} \left[ e^{2 \alpha n \beta^{(1,i)}_n(\delta, T)} \right]=0.
\end{equation}
\end{lemma}

{\it Proof:} We treat $i=1, 2,$ and $3$ separately. 
For $i=1$ we  use the known expression for the Poisson mgf, so as to obtain
\begin{equation}\label{CCB1}
\frac{1}{n}\log\mathbb{E} \left[ e^{2 \alpha n \beta^{(1,1)}_n(\delta, T)}  \right] =\frac{1}{n}\log\mathbb{E} \left[ e^{2\alpha \bar A_n(\delta)} \right] = \lambda \delta (e^{2\alpha} -1) \to 0, \quad \text{as }\delta\downarrow 0.
%\quad \mbox{ and } \quad \mathbb{E} \left[ \exp \left\{ \alpha \sum_{i=1}^{nf_0} B_i \right\} \right]= (\delta C e^\alpha + 1 - \delta C)^{n f_0}, 
\end{equation}
%where the second equality holds for $\delta \leqslant 1/C$.
For $i=2$ we first observe that 
\[
\sup_{t \in [0,T]} \sum_{i=1}^{nf_0} \bs 1 \{S^\circ_i \in [t, t+\delta] \} \leqslant \max_{k \in \{0, 1, \dots, T/\delta\}} \sum_{i=1}^{n f_0}  \bs 1 \{S^\circ_i \in [k\delta, k\delta+2\delta] \}.
\]
Recall that it was assumed that the density $h^\circ(\cdot)$ of the residual sojourn times $S^\circ_i$ is uniformly bounded by some finite constant $C$. 
Consequently, for any $k$, \[
\sum_{i=1}^{n f_0}  \bs 1 \{S^\circ_i [k\delta, k\delta+2\delta] \}\stackrel{\rm st}{\leqslant}B\sim {\mathbb B}{\rm in}(n f_0, 2 C \delta)\] 
(where $\delta$ is sufficiently small to guarantee  $2C\delta<1$). 
By \cite[Theorem 2.3]{McD98}, which is effectively a Chernoff inequality, we have
\[
\mathbb{P}(B \geqslant nf_0 (2C\delta + \varepsilon)) \leqslant \exp\left\{-nf_0\left((2C\delta+\varepsilon)\log\left(1+\frac{\varepsilon}{2C\delta}\right)-\varepsilon\right)\right\}
\]
 for any $\varepsilon>0$.
 Upon combining the above bound, we thus conclude that
 \begin{align*}
 \mathbb{E} \left[ e^{2 \alpha n \beta^{(1,2)}_n(\delta, T)} \right] &\leqslant \mathbb{E} \left( \exp \left\{ 2\alpha  \max_{k \in \{0, 1, \dots, T/\delta\}} \sum_{i=1}^{n f_0}  \bs 1 \{S^\circ_i \in [k\delta, k\delta+2\delta] \} \right\}  \right) \\
 &\leqslant e^{2\alpha n f_0 (2C\delta +\varepsilon)}+ e^{2\alpha n f_0}\frac{T}{\delta} \mathbb{P}(B \geqslant n f_0(2C\delta+\varepsilon)) \\
 &\leqslant e^{2\alpha n f_0 (2C\delta +\varepsilon)} +e^{2\alpha n f_0}\frac{T}{\delta}\exp\left\{-nf_0\left((2C\delta+\varepsilon)\log\left(1+\frac{\varepsilon}{2C\delta}\right)-\varepsilon\right)\right\} \\
 &\leqslant 2 \max \left\{ e^{2\alpha n f_0 (2C\delta +\varepsilon)}, e^{2\alpha n f_0}\frac{T}{\delta} \exp\left\{-nf_0\varepsilon \left(\log\left(1+\frac{\varepsilon}{2C\delta}\right)-1\right)\right\}  \right\};
 \end{align*}
 in the second inequality we distinguish between the contributions of the events $\{B < n f_0(2C\delta+\varepsilon)\}$ and $\{B \geqslant n f_0(2C\delta+\varepsilon)\}$, respectively.
Consequently,
\begin{align*}
&\lim_{\delta \downarrow  0} \limsup_{n \to \infty} \frac{1}{n} \log \mathbb{E} \left[ e^{2 \alpha n \beta^{(1,2)}_n(\delta, T)}\right] \\
&\qquad  \leqslant \lim_{\delta \downarrow  0} \max \left\{ 2 \alpha f_0(2C\delta+\varepsilon), \; 2 \alpha f_0 - f_0 \varepsilon  \left(\log\left(1+\frac{\varepsilon}{2C\delta}\right)-1\right) \right\} =2\alpha f_0 \varepsilon.
\end{align*}
where in the final step we observe that for any $\varepsilon >0$ the second term in the maximum converges to $-\infty$ as $\delta \downarrow  0$. Since $\varepsilon$ is an arbitrary constant, we obtain \eqref{CB11} for $i=2$ by taking $\varepsilon \downarrow 0$.

We conclude with the analysis corresponding to $i=3$. We let $A^\star _n(\cdot)$ be a sequence of Poisson processes on $\mathbb{R}$ with intensity $\lambda n$ which are independent of everything else, and for $b <c$ let $A_n^\star [b,c]$ denote the number of points contained in the interval $(b,c)$. Observe that
\begin{align}
    \sup_{t \in [0,T]} \left\{ \int^t_0 {\rm d}u \sum_{i=1}^{A_n(t)} \bs{1} \{\tau_i=u\} \bs 1 \{S_i \in [t-u, t-u+\delta] \} \right\} &\stackrel{\rm st}{\leqslant} \sup_{t \in [0,T]}   A^\star _n[t,t+\delta]  \label{PPP} \\
    &\leqslant \max_{k \in \{0, 1, \dots, T/\delta\}}  A_n^\star[k\delta, k\delta+ 2 \delta], \nonumber
\end{align}
where for simplicity we assume that $T/\delta$ is an integer.
To understand the validity of \eqref{PPP}, observe that the departure process of clients
{when there are initially no clients present is dominated by the departure process of clients when there are initially a stationary number of clients present. 
Equation \eqref{PPP} then follows by from the known property that the latter is a Poisson process with intensity $\lambda n$.}
% To understand the validity of \eqref{PPP}, observe that the departure process of clients is obtained by taking $A_n(\cdot)$ (i.e., the arrival process --- a Poisson process on $\mathbb{R}_+$) and shifting each point to the right by an iid amount (the sojourn time); this process is stochastically dominated by a Poisson process on $\mathbb{R}$ with all points shifted to the right by an iid amount. We can then note that shifting the points to the right by iid amount does not change the distributional properties of the Poisson process.\footnote{\small MM: Should we formalize this? It is known that the departure process of the M/G/$\infty$ is a Poisson process. And I don't fully understand what stochastic dominance is meant. PB: This is true when the initial number of clients in the M/G/$\infty$ is given by the stationary distribution of the queue. When the queue is initially empty then there are fewer departures. This is what essentially what I was trying to explain here: the number of departures from an M/G/$\infty$ which is initially empty is stochastically dominated by the number of departures from a stationary M/G/$\infty$ queue, which is equivalent to a Poisson process.}

If $Z$ is a Poisson random variable with mean $a$ we obtain (via a Chernoff bound; see for instance \cite[Example 7.3]{Ros11}) that
\[
\mathbb{P}( Z - a \geqslant x) \leqslant \exp \left\{ -x \left( \log \left(1+\frac{x}{a} \right)-1 \right)-a \log \left(1 + \frac{x}{a}\right) \right\} .
\]
Thus,
\begin{equation*}
    \mathbb{P}(A^\star _n[0, 2 \delta] \geqslant n(\varepsilon+2\lambda\delta)) \leqslant \exp \left\{ -n \varepsilon \left( \log \left(1+\frac{\varepsilon}{2 \lambda \delta} \right)-1 \right)-2n \lambda \delta \log \left(1+\frac{\varepsilon}{2 \lambda \delta}\right) \right\} .
\end{equation*}
Consequently, for any $\varepsilon >0$,
\begin{align}
     \mathbb{E} \left[ e^{2 \alpha n \beta^{(1,3)}_n(\delta, T)} \right] &\leqslant \mathbb{E} \left( \exp\left\{2\alpha \max_{k \in \{0,1, \dots, T/\delta\}} \left\{ A^\star _n[k\delta,k\delta+ 2 \delta] \right\} \wedge n \right\} \right) \nonumber \\
    &\leqslant e^{2\alpha n(\varepsilon+2\lambda \delta)} + e^{2\alpha n} \frac{T}{\delta} \mathbb{P} ( A^\star _n[0,2 \delta] \geqslant n(\varepsilon+2\lambda \delta)) \nonumber \\  
    &\leqslant  e^{2\alpha n(\varepsilon+2\lambda \delta)} + e^{2\alpha n}\frac{T}{\delta} \exp \left\{ -n \varepsilon \left( \log \left(1+\frac{\varepsilon}{2 \lambda \delta} \right)-1 \right)-2n \lambda \delta \log \left(1+\frac{\varepsilon}{2 \lambda \delta}\right) \right\} \nonumber  \\
    &\leqslant 2 \max \left\{ e^{2\alpha n (\varepsilon+2\lambda \delta)}, e^{2\alpha n}\frac{T}{\delta} \exp \left\{ -n \varepsilon \left( \log \left(1+\frac{\varepsilon}{2 \lambda \delta} \right)-1 \right) \right\} \right\}. \label{CCB2} 
\end{align}
We then obtain
\begin{align*}
\lim_{\delta \downarrow  0} \limsup_{n \to \infty} \frac{1}{n} \log \mathbb{E} \left[ e^{2 \alpha n \beta^{(1,3)}_n(\delta, T)}\right]  &\leqslant \lim_{\delta \downarrow  0}  \max\left\{ 2\alpha(\varepsilon + 2 \lambda \delta), \; 2\alpha -  \varepsilon \left( \log \left( 1 + \frac{\varepsilon}{2 \lambda \delta} \right) -1 \right) \right\}  \\
&=2 \alpha \varepsilon.
\end{align*}
Since $\varepsilon$ is an arbitrary constant, we have obtained \eqref{CB11} for $i=3$ by taking $\varepsilon \downarrow 0$. \hfill$\Box$

%\begin{align*}
%\lim_{\delta \downarrow  0} &\limsup_{n \to \infty} \frac{1}{n} \log \mathbb{E} \left[ e^{2 \alpha n \beta_n^{(1)}(\delta, T)} \right] \\
%&\leqslant \lim_{\delta \downarrow  0} \limsup_{n \to \infty} \bigg[ \lambda \delta( e^{\alpha} -1) + f_0 \log ( \delta C e^{\alpha} + 1 - \delta C) + \frac{\log 2}{n} \\
%&\qquad + \max \left\{ \alpha \varepsilon, \; \alpha + \frac{1}{n}\log \frac{T}{2\delta} - \varepsilon \left( \log \left( 1 + \frac{\varepsilon}{2 \lambda \delta} \right) -1 \right)  \right\}  \bigg]  \\
%&=\lim_{\delta \downarrow  0} \bigg[\lambda \delta( e^{\alpha} -1) + f_0 \log ( \delta C e^{\alpha} + 1 - \delta C) + \max \left\{ \alpha \varepsilon, \; \alpha - \varepsilon \left( \log \left( 1 + \frac{\varepsilon}{2 \lambda \delta} \right) -1 \right)  \right\}\bigg] \\
%&= \alpha \varepsilon,
%\end{align*}
%where in the final step we observe that for any $\varepsilon >0$ the second term in the maximum converges to $-\infty$ as $\delta \downarrow  0$. Since $\varepsilon$ is an arbitrary constant we obtain \eqref{CB1} by taking $\varepsilon \to 0$.

% as $n\to \infty$ this expression becomes 
% \[
% \max \left\{ \frac{\varepsilon}{2}, \frac{1}{2}-\frac{1}{2}\varepsilon \left( \log\left( 1 + \frac{\varepsilon}{2 \lambda \delta} \right) -1 \right)-\lambda \delta \log \left(1 + \frac{\varepsilon}{2 \lambda \delta} \right) \right\} 
% \]
% and as $\delta \downarrow  0$ it becomes $\varepsilon/2$ because for any $\varepsilon>0$ the second term in the above maximum converges to $-\infty$.

We continue by analysing the contribution corresponding to $\beta^{(2)}_n(\delta, T)$.

\begin{lemma} \label{LEM2} It holds that
\begin{equation}\label{CB2}
\lim_{\delta \downarrow  0} \limsup_{n \to \infty} \frac{1}{n} \log \mathbb{E} \left[ e^{2 \alpha n \beta^{(2)}_n(\delta, T)} \right]=0.
\end{equation}
\end{lemma}
{\it Proof:}
We distinguish between two cases: $A_n(T) \leqslant nK$ and $A_n(T) > nK$, where $K$ is an arbitrary constant that we will select later to suit our purposes. In addition, we consider two sub-cases when $A_n(T) \leqslant nK$: when $\bar Y_n(\delta,T) \leqslant nK'$ and when $\bar Y_n(\delta,T) > nK'$, where $\bar Y_n(\delta,T) $ is as defined in \eqref{eq:defy}, and $K'$ denoting  an arbitrary constant.  We have
\begin{align}
\mathbb{E} \left[ e^{2 \alpha n \beta^{(2)}_n(\delta, T)} \right] &\leqslant \mathbb{E} \left[ e^{2 \alpha n \beta^{(2)}_n(\delta, T)} \,\big|\, A_n(T) \leqslant nK \right] + e^{2 \alpha n} \mathbb{P}( A_n(T) > n K) \nonumber \\
&\leqslant \mathbb{E} \left[  e^{2 \alpha n \beta^{(2)}_n(\delta, T)} \,\big| \, A_n(T) \leqslant n K, \bar Y_n(\delta,T) \leqslant n  \delta K'\right] \nonumber \\
&\qquad + e^{2 \alpha n} \mathbb{P} \left( \bar Y_n(\delta,T) > n \delta K' \, \big| \, A_n(T) \leqslant n K\right) + e^{2 \alpha n} \mathbb{P}( A_n(T) > n K). \nonumber \\
&\leqslant B_1(\delta,K,K') + B_2(\delta,K,K')+B_3(\delta,K), \nonumber 
%&\leqslant 3 \max \{ B_1(\delta,K,K'), B_2(\delta,K,K'), B_3(\delta,K) \} \nonumber.
\end{align}
% \begin{align}
% \mathbb{E} \left[ e^{\gamma_n(\delta, \alpha, T)} \right] &\leqslant \mathbb{E} \left[ e^{\gamma_n(\delta, \alpha, T)} | A_n(T) \leqslant nK \right] + e^{2 \alpha n} \mathbb{P}( A_n(T) > n K) \nonumber \\
% &\leqslant \mathbb{E} \left[  e^{\gamma_n(\delta, \alpha, T)} \bigg| A_n(T) \leqslant n K, \sum_{i=1}^{TC} Y_i \leqslant n  \delta K_2\right] + e^{2 \alpha n} \mathbb{P} \left[ \sum_{i=1}^{TC} Y_i \leqslant n \delta K_2 \bigg| A_n(T) \leqslant n K\right] \nonumber \\
% &\qquad  + e^{2 \alpha n} \mathbb{P}( A_n(T) > n K). \nonumber \\
% &\leqslant \mathbb{E}\left[ e^{\alpha \bar A_n(\delta)} \right] e^{\alpha r \delta n f_0} \mathbb{E} \left[ \exp \left\{ \alpha \sum_{i=1}^{n f_0} B_i \right\}  \right] e^{\alpha  r\delta nK} e^{\alpha n\delta K_2 } \nonumber \\
% &\quad  \times \mathbb{E}\left[ \left. \exp \alpha \left\{  \int_0^t {\rm d} u \sum_{i=1}^{A_n(T)} \boldsymbol 1 \{ \tau_i=u\} \boldsymbol  1 \{ S_i \in [t-u,t-u+\delta] \} \right\} \right| A_n(T)=nK\right] \nonumber \\
% &\quad + e^{2 \alpha n} \mathbb{P} \left[ \sum_{i=1}^{TC} Y_i \leqslant n \delta K_2 \bigg| A_n(T)= n K\right] + e^{2 \alpha n} \mathbb{P}( A_n(T) > n K). \nonumber \\
% %&\quad + \mathbb{E}\left[ e^{\alpha \bar A_n(\delta)} \right] e^{\alpha r \delta n f_0} e^{2\alpha n} \mathbb{P}(A_n(T) > nK). 
% &=: B_1(\delta,K,K_2) + B_2(\delta,K,K_2)+B_3(\delta,K) \nonumber \\
% &\leqslant 3 \max \{ B_1(\delta,K,K_2), B_2(\delta,K,K_2), B_3(\delta,K) \} \nonumber.
% \end{align}
where 
\begin{align*}
    B_1(\delta,K,K') &:=e^{2\alpha n \delta( r f_0 + r K + K')},\\
    B_2(\delta,K,K') &:=e^{2 \alpha n} \mathbb{P} \left( \bar Y_n(\delta,T) > n \delta K' \,\big|\, A_n(T)= n K\right),\\
    B_3(\delta,K) &:= e^{2 \alpha n} \mathbb{P}( A_n(T) > n K).
\end{align*}

To verify \eqref{tightnessC2} we deal with each terms $B_1(\delta,K,K')$, $B_2(\delta,K,K')$,  and $B_3(\delta,K)$ separately. The first term is straightforward:
clearly,
\[
\lim_{\delta \downarrow  0} \limsup_{n \to \infty} \frac{1}{n} \log B_1(\delta,K,K_2) = 0
\]
for any choice of the parameter values.
Also the third term, $B_3(\delta, K,K')$, can be dealt with in a direct fashion, relying on Cram\'er's theorem for the sum of independent Poisson random variables. In particular, we use the fact that if $\mathbb{E}(A_n(T)/n)\equiv\lambda T<K$, then
\[
\lim_{n \to \infty} \frac{1}{n} \log \mathbb{P}(A_n(T) \geqslant nK) = -K \log \frac{K}{\lambda T} + K - \lambda. 
\]
We now have 
\[
 \lim_{\delta \downarrow  0} \limsup_{n \to \infty} \frac{1}{n} \log B_3(\delta, K) =2\alpha -K \log \frac{K}{\lambda T}+ K - \lambda.
\]
For any $\alpha>0$ we can choose $K$ large enough to ensure that this terms is negative.

Finally, we analyze the second term $B_2(\delta, K,K')$, again applying Cram\'er's theorem.
First observe that under the condition $A_n(T)=nK$, $\bar A_n(\delta,T)$ has a Poisson distribution with mean $n \delta \nu (f_0 + K)$. Using the thinning property of a Poisson process, we see that $\bar Y_n(\delta,T)$ has the same distribution as 
\[
\sum_{j=1}^{n} \sum_{i=1}^{Z_j} Y_{i,j}
\]
where the $\{Y_{i,j} \}_{i,j\in{\mathbb N}}$ are independent random variables with mgf $\beta(\cdot)$, and $\{ Z_j \}_{j \in{\mathbb N}}$ are independent Poisson random variables with mean $\delta \nu (f_0 + K)$. For any $j \in{\mathbb N}$,  
\begin{equation}\label{mgftig}
{\mathbb E}\left[e^{\theta\sum_{i=1}^{Z_j} Y_{i,j} }\right]=
\exp \left\{ \delta \nu (f_0 + K) ( \beta (\theta)-1) \right\} :=J_\delta(\theta,K).
\end{equation}
Applying Cram\'er's theorem for sums of iid random variables, we obtain
\begin{align*}
\lim_{n \to \infty}\frac{1}{n} \log \mathbb{P} \left(  \bar Y_n(\delta,T)> n \delta K' \,\big|\, A_n(T)= n K\right) &= \lim_{n \to \infty} \frac{1}{n} \log \mathbb{P} \left( \frac{1}{n \delta}\sum_{j=1}^{n} \sum_{i=1}^{Z_j} Y_{i,j} > K' \right) \\
&=-I_\delta(K,K') ,
\end{align*}
where $I_\delta(K,K'):= \sup_{\theta \in \mathbb{R}} ( \theta K'  - \log J_\delta(\theta,K) )$
is the Legendre transform of \eqref{mgftig}. Now note that $I_\delta(K,K') \to  \infty$ as $\delta \downarrow  0$ for any $K, K' >0$.
We thus conclude
\[ 
\lim_{\delta \downarrow  0} \limsup_{n \to \infty} \frac{1}{n} \log B_{2}(\delta, K,K') = \lim_{\delta \downarrow  0} \left[ 2 \alpha - I_\delta(K,K') \right]= -\infty.
\]

Consequently, for any $\alpha >0$ we can choose $K$ and $K'$ such that 
\begin{align*}
\lim_{\delta \downarrow  0} &\limsup_{n \to \infty} \frac{1}{n} \log \mathbb{E} \left[ e^{2 \alpha n \beta^{(2)}_n(\delta, T)}\right] \\
&\leqslant  \lim_{\delta \downarrow  0} \limsup_{n \to \infty} \frac{1}{n} \log \big(3 \max \left\{ B_1(\delta, K, K'), B_2(\delta, K, K'), B_3(\delta, K) \right\} \big)=0.
\end{align*}
We have thus verified  the claim.
\hfill$\Box$

\vb

Lemmas \ref{LEM1} and \ref{LEM2} entail that we have verified Condition \eqref{tightnessC2}. As we had already verified Condition \eqref{tightnessC1}, we have finished the proof of Theorem \ref{thm:main}.

Now that we have proven that the bivariate process $(\bar F_n(\cdot), \bar G_n(\cdot))$ is exponentially tight, we finish this subsection by  showing (i-a).
\begin{lemma}
We have $I^{\mathcal{X}}_{[0,T]}(f,g)=\infty$ when $f$ or $g$ is not absolutely continuous.
\end{lemma}
\begin{proof}
Recall that 
\[
I_{\bs t}(\bs f,\bs g)=\sup_{\bs\omega, \bs\theta} \left(  \sum_{j=1}^d \omega_j f_j+
\sum^d_{j=1} \theta_j g_j - f_0 \log M^-_{\bs t}(\bs\omega, \bs\theta) - \log M^+_{\bs t}(\bs\omega, \bs\theta) \right). 
\]
It thus suffices to show that if $f$ or $g$ are not absolutely continuous then there exist sequences $\{\bs t^n \}$, $\{\bs \omega^n\}$, and $\{\bs \theta^n\}$  such that $I_{\bs t^n}(\bs f^n,\bs g^n)\to \infty$, where $f^n_j=f(t^n_j)$ and $g^n_j=g(t^n_j)$.

We start with the case  that $f$ is not absolutely continuous. This means that there exists $\delta>0$ and $\{s_1^n<u^n_1 \leqslant \dots \leqslant s^n_{k_n} < u^n_{k_n}\}$ such that $\sum_{\ell=1}^{k_n}(u^n_\ell - s^n_{\ell}) \to 0$, while $\sum_{\ell}^{k_n} | f(u_\ell^n)-f(s_\ell^n)| \geqslant \delta$.
Let $\bs t^n=(t^n_i)_{i\in\{1,\ldots, 2k_n\}}$ where, for $\ell\in\{1,\ldots,k_n\}$, we have $t_{2\ell-1}=s^n_\ell$ and $t_{2\ell}=u^n_\ell$. 
In addition, let $\bs \omega^n = (\omega^n_i)_{1\leqslant i \leqslant 2k_n}$ where, for $\ell\in\{1,\ldots,k_n\}$, we have 
\[
\omega_{2\ell-1}=\alpha(1-2{\bs 1}\{f(t_{2\ell})\geqslant f(t_{2 \ell-1})\} \quad \mbox{and}  \quad \omega_{2\ell}=\alpha(2{\bs 1}\{f(t_{2\ell})\geqslant f(t_{2 \ell-1})\}-1),
\]
and $\bs \theta^n=\bs 0$.
Then 
\[
\sum_{j=1}^{2k_n} \omega^n_j f^n_j = \alpha \sum_{\ell=1}^{k_n} |f(u^n_\ell)-f(s^n_\ell)| \geqslant \alpha \delta.
\]
Since $\alpha$ is an arbitrary constant the result is proved if we can show that
\[
f_0 \log M^-_{\bs t^n}(\bs \omega^n, \bs 0) + \log M^+_{\bs t^n}(\bs \omega^n, \bs 0) \to 0, \quad \forall \, \alpha >0. 
\]
Due to the particular choice of $\bs \omega^n$, $M^-_{\bs t^n}(\bs \omega^n, \bs 0)$ and $M^+_{\bs t^n}(\bs \omega^n, \bs 0)$ only capture changes in the client population size during $\cup_{\ell=1}^{k_n} [u^n_{\ell}, s^n_{\ell})$. 
Since this interval is vanishing many of the arguments used to establish exponential tightness carry over (in this case Lemma \ref{LEM1} specifically), and hence we will be brief with our explanations.
In particular, we have
\begin{align}
    M^-_{\bs t^n}(\bs \omega^n, \bs 0) &\leqslant 1+e^{\alpha}\mathbb{P}\left(\tau^\circ \in \bigcup_{\ell=1}^{k_n} [u^n_{\ell}, s^n_{\ell})\right) \to 1  \label{eq:Mt1} \\
    M^+_{\bs t^n}(\bs \omega^n, \bs 0) &\leqslant \exp\left\{ \lambda \sum_{\ell=1}^{k_n} (u^n_\ell-s^n_\ell)(e^{2\alpha}-1) \right\} \to 1,\label{eq:Mt2}
\end{align}
where $\tau^\circ$ is a random variable with density $h^\circ(\cdot)$; the convergence in \eqref{eq:Mt1} follows from the existence of a density {(recall the setup in Section \ref{Sec:model})},
% \footnote{MM: hm, now I'm puzzled. Hadn't we claimed that this is not necessary?}
whereas {the inequality in \eqref{eq:Mt2} follows from the observation that when clients arrive according to a Poisson process with rate $\lambda$ then their departures are dominated by a Poisson process with rate $\lambda$ (recall the explanation after \eqref{PPP}) in combination with H\"{o}lder's inequality.}
% \footnote{MM: `compared with' sound vague. Can't we always compare? I guess we have to be more specific.} 
% to a Poisson process with arrival rate $\lambda$ and H\"{o}lder's inequality. 

When $g$ is not absolutely continuous the arguments are similar (i.e., we let $\bs \theta^n$ play the role of $\bs \omega^n$ above), but to establish
\[
f_0 \log M^-_{\bs t^n}(\bs 0, \bs \theta^n) + \log M^+_{\bs t^n}(\bs 0, \bs \theta^n) \to 0, \quad \forall \, \alpha >0
\]
we now follow the same line of reasoning that led to Lemma \ref{LEM2}.
\end{proof}

\subsubsection{Alternative expression for the action functional}\label{Sec:AF}

In this subsection we provide an alternative expression for the action functional $I_{[0,T]}(f,g)$, which may be attractive for computational purposes. 
The main idea is that we decompose the action functional based on the observation that $F_n(\cdot)$ modulates $G_n(\cdot)$ and as such evolves independently. This informally means that we can write $I_{[0,T]}(f,g)$ as the action functional describing the cost of $\bar F_n(\cdot)$ being close to $f(\cdot)$, increased by the the action functional describing the cost of $\bar G_n(\cdot)$ being close to $g(\cdot)$ conditional on $\bar F_n(\cdot)$ being close to $f(\cdot)$. Below we provide expressions for both components featuring in this decomposition. The same type of decomposition has appeared, in different contexts, in for instance \cite{GANG,LIP}.

We start by evaluating the action functional of $\bar F_n(\cdot)$ for the path $f(\cdot)$.
Note that, in passing, we established a `marginal LDP' for the client-population size  only (i.e., not including the net aggregate claim process). From the joint LDP of the client-population-size process and the net aggregate claim process, we find that the corresponding rate function reads
\begin{align*}
 I _{[0,T]}(f):=\sup_{\omega(\cdot)} \bigg\{
 \int_0^T\omega(s) f(s)\,{\rm d}s\:
-&\:f_0\log\left(\int_0^T h^\circ(u)\,e^{\Omega(u)}\,{\rm d}u +\bar h^\circ(T)\,e^{\Omega(T)}\right)\\
 \:
-&\:\lambda\int_0^T\left(
 \int_s^T h(r-s) \frac{e^{\Omega(r)}}{e^{\Omega(s)}}{\rm d}r
 +\bar h(T-s) \frac{e^{\Omega(T)}}{e^{\Omega(s)}} -1
 \right){\rm d}s
    \bigg\}.
\end{align*}
A complication of this optimization problem is that the argument $\omega(\cdot)$ also appears as its integrated version $\Omega(\cdot)$. However, by applying integration by parts, $\omega(\cdot)$ can be eliminated from this variational problem, so that it is written in terms  $\Omega(\cdot)$ only. Indeed, an equivalent variational problem is
\begin{align*}
 I _{[0,T]}(f)=\sup_{\Omega(\cdot)} \bigg\{
 \Omega(T) f(T)&-
 \int_0^T\Omega(s) f'(s)\,{\rm d}s-f_0\log\left(\int_0^T h^\circ(u)\,e^{\Omega(u)}\,{\rm d}u +\bar h^\circ(T)\,e^{\Omega(T)}\right)\\
&- \lambda\int_0^T\left(
 \int_s^T h(r-s) \frac{e^{\Omega(r)}}{e^{\Omega(s)}}{\rm d}r
 +\bar h(T-s) \frac{e^{\Omega(T)}}{e^{\Omega(s)}} -1
 \right){\rm d}s
    \bigg\},
\end{align*}
with $\Omega(0)=0$.
For ease rewriting $z(s):=\exp( \Omega(s))$, this further reduces to
\begin{align}\nonumber
 I _{[0,T]}(f)=\sup_{z(\cdot)} \bigg\{
 &\log z(T) f(T)-
 \int_0^T\log z(s) f'(s)\,{\rm d}s-f_0\log\left(\int_0^T h^\circ(u)\,z(u)\,{\rm d}u +\bar h^\circ(T)\,z(T)\right)\\
 &- \lambda\int_0^T\left(
 \int_s^T h(r-s) \frac{z(r)}{z(s)}{\rm d}r
 +\bar h(T-s) \frac{z(T)}{z(s)} -1
 \right){\rm d}s
    \bigg\}. %= I_{[0,T]}(f)
    \label{VP}
\end{align}
%Due to the concave structure underlying this Legendre transform, we can find the optimising $z(\cdot)$ by solving the corresponding first order condition. Differentiation with respect to $z(v)$ yields:
%\begin{align*}
%    0=-\frac{f'(v)}{z(v)} +A_z(v) + \frac{B_z(v)}{(z(v))^2},
%\end{align*}
%with
%\[A_z(v) = -\frac{f_0 \,h^\circ(v)}{\displaystyle 
%\int_0^T 
%h^\circ(u)z(u)\,{\rm d}u +\bar h^\circ(T)z(T)}+\int_0^v \frac{h(v-s)}{z(s)}{\rm d}s
%\]
%and
%\[B_z(v):=-\int_v^T h(s-v)z(s)\,{\rm d}s+\lambda\,\bar h(T-v)Z(T)
%.\]
%We can thus find, by solving the above quadratic equation, solve $z(v)$ in terms of $A_z(v)$ and $B_z(v)$, but notice that $A_z(v)$ and $B_z(v)$ involve the full function $(z(u))_{u\in[0,T]}$. Considering $z$ at an equidistant grid $w_i:=z(i\Delta)$ for $i=1,\ldots,T/\Delta$, we can  
%numerically find these $w_i$ by solving the corresponding $T/\Delta$ equations. 
%Alternatively, one could approximate the objective function in \eqref{VP} by looking at the equidistant grid $w_i:=z(i\Delta)$, and numerically solve the maximization problem (where we note the objective function is concave).  
%
%For the special case of exponential durations, more explicit expressions have been derived. As can be found in e.g.\  \cite{SW}, if $h(r)=h^\circ(r) = \mu e^{-\mu r}$, then
%\[I_{[0,T]}(f) = \int_0^T J_{f(s)}(f'(s))\,{\rm d}s,\]
%with 
%\[J_x(u) := \sup_\theta \big(\theta u - \lambda (e^\theta-1) - x\mu (e^{-\theta}-1)\big).\]
%Observe that $J_x(u)$ can be calculated explicitly; the first order condition amounts to solving a quadratic equation.

Conditional on the path $f(\cdot)$ describing the evolution of the client-population size, we now focus on the action functional of the reserve process $G_n(\cdot)$. Given that $\bar F_n(\cdot)$ is close to $f(\cdot)$, a path $g(\cdot)$ of the reserve process has, between $0$ and $T$, rate function
\[I_{[0,T]}(g\,|\,f) = \int_0^T K_{f(s)}(g'(s))\,{\rm d}s,\:\:\:\:\mbox{where}\:\:\:
K_x(u):= \sup_\theta(\theta u -x\,\varphi(\theta));\]
this relation can be considered as a version of Mogulskii's theorem corresponding to the setting of a random walk of which the increments have a deterministically time-varying distribution. (Informally, the rationale behind the expression for $I_{[0,T]}(g\,|\,f)$ is that, by `locally applying Cram\'er's theorem', it equals
\[\lim_{\Delta\downarrow 0} \Delta \sum_{i=0}^{T/\Delta} \sup_\theta\left(\theta g'(i\Delta)-\log e^{\theta r f(i\Delta)}
-\log \left(\sum_{k=0}^\infty e^{-\nu f(i\Delta)}\frac{(\nu f(i\Delta))^k}{k!}\big(\beta(\theta)\big)^k\right)\right);\]
evaluating this Riemann sum yields the expression for $I_{[0,T]}(g\,|\,f)$ that was postulated above.)
Then $I _{[0,T]}(f,g)$ can be computed through the relation
\[
I _{[0,T]}(f,g) =  I _{[0,T]}(f)+ I_{[0,T]}(g\,|\,f).
\] 

\bibliographystyle{plain}

\appendix
\section{Computational techniques}\label{Sec:Comp}
% \footnote{PB: This section needs to be streamlined. I was also not sure if this should be a section or an Appendix. It does not fit in Section \ref{Sec:proofs} as it is not part of a proof. MM: I agree it should be an appendix. \textcolor{black}{PB: I have streamlined the Appendix.}}
{In this appendix we describe a numerical method to solve the variational problem described by Equations \eqref{VPA1} and \eqref{VPA2}. Recall that our goal is to find the most likely path in the set 
\[
\mathscr{H}_t= \{(f,g):(f(t),g(t)) \in \mathscr{R} \}, \quad \text{where} \quad \mathscr{R}=[0,\infty)\times B,
\]
i.e., the path $(f\s, g\s)$ such that $I_{[0,T]}(f\s,g\s)=\varrho(t):=\inf_{f,g \in \mathscr{H}_t} I_{[0,T]}(f,g)$.
To this end we write $f\s \equiv f\s(t)$ and $g\s \equiv g\s(t)$, and find $\omega\s$ and $\theta\s$ such that 
\[
\varrho(t) = I_t(f\s, g\s) = \omega\s f\s + \theta\s g\s - N_t(\theta\s,\omega\s),
\]
where $N_t(\omega,\theta):=f_0 \log M_t^- + \log M^+_t(\omega, \theta)$, $I_t(\cdot, \cdot)$ is the rate function of the one-point LDP given in Proposition \ref{P1}, and $\omega\s$ and $\theta\s$ are the optimising values of $\omega$ and $\theta$.
We will argue that this computation can be used as the basis for an efficient technique that yields {\it the full most likely path}; cf.\ the results for most likely workload paths in queues fed by many iid sources, as developed in e.g.\ \cite{WIS}.}

{Fix $s \in [0,t]$. By the contraction principle, applying the bivariate LDP, we have
\[
I_t(f\s,g\s) = \inf_{f \geq 0, g \in \mathbb{R}} I_{s,t}((f,f\s),(g,g\s)).
\]
We wish to identify the optimising $f$ and $g$ in the right-hand-side, which can be interpreted as $f\s(s)$ and $g\s(s)$. 
The optimising arguments in the definition of $I_{s,t} ((f,f\s),(g,g\s))$ are $((0,\omega\s),(0,\theta\s))$:
\[I_{s,t} \big((f\s(s),f\s),(g\s(s),g\s)\big)=  0\cdot f\s(s)+
\omega\s f\s+0\cdot g\s(s) + \theta\s g\s-N_{s,t}\big((0,\omega\s),(0,\theta\s)\big) ,\]
with $N_{s,t}({\bs \omega},{\bs \theta}):=f_0 \log M_{s,t}^-({\bs \omega},{\bs \theta})+\log M_{s,t}^+({\bs \omega},{\bs \theta})$.
As a consequence,
\begin{align*}g\s(s) &= \left.\frac{\partial}{\partial \theta_1}N_{s,t\s}({\bs \omega},{\bs \theta})\right|_{({\bs \omega},{\bs \theta})=
((0,\omega\s),(0,\theta\s))
},\\
f\s(s) &= \left.\frac{\partial}{\partial \omega_1}N_{s,t\s}({\bs \omega},{\bs \theta})\right|_{({\bs \omega},{\bs \theta})=
((0,\omega\s),(0,\theta\s))
}.
\end{align*}
As we can do this for any $s$, we have found a way to evaluate the full most likely path. 
}

\section{The limiting value of $E_1(a,T)$}\label{app1}

In this appendix we present the calculations that lead to \eqref{ECT}. As indicated in the main text, we consider the case that clients remain at the insurance firm for an exponentially distributed length of time with mean $1/\mu$.
Note that, due to the exponential sojourn times, $(F_n(t), G_n(t))_{t \in [0,T]}$ is a Markov process.
With
\[I_{t_1,t_2}((f_1,f_2),{\mathbb R}^2):=
\inf_{(g_1,g_2)\in{\mathbb R}^2} I_{t_1,t_2}((f_1,f_2),(g_1,g_2)),\]
the rate associated to the client-population-size process is
\[
I_{[0,T]}(f^{(\star, T)}) = \sum_{i=1}^{T/\Delta} I_{\Delta i, \Delta(i+1)} ( (f^{(\star, T)}(\Delta i), f^{(\star, T)}(\Delta (i+1))),\mathbb{R}^2)
\]
for any $0< \Delta< T$ where we have applied to the contraction principle to obtain equality with the rate of the finite-dimensional LDP, and the Markov property to decompose the rate function of the finite-dimensional LDP into a sum.
{Let $\varepsilon =a - \bar g(T)$. For ease of exposition we will tacitly assume that $a < \bar g(T)$, i.e., there is an unusually large surplus at time $T$.}
Let $t=\Delta i$ and ${\rm d}t=\Delta$.
The additional clients that can be attributed to the interval $[t, t+{\rm d}t)$ are
\[
{\rm d}t\, a(t) := f^{(\star, T)}(t+{\rm d}t) - \bar f^{(\star, T)}(t+{\rm d}t),
\]
where $\bar f^{(\star, T)}(t+{\rm d}t)$ is the expected client-population size at time $t+{\rm d}t$ given the client population is $f^{(\star,T)}(t)$ at time $t$.
{The expected total capital generated by each additional client that arrived in $[t,t+{\rm d}t)$ by time $T$ (in the conditioned process) is approximately
\[
\int^{T-t}_0 e^{-\mu x}(r - \nu \bar m)\,{\rm d}t = \frac{r-
\nu \bar m}{\mu}(1-e^{-\mu(T-t)}),
\]
where this approximation holds for small ${\rm d}t$ and $\varepsilon$, and uses the fact that for $\varepsilon$ small, clients in the conditioned process generate claims in a similar manner as in the unconditioned process.
Consequently, the total capital that can be attributed to the additional clients that arrived in the interval $[t,t+{\rm d}t)$ is approximately
\begin{equation}\label{CCO1}
    {\rm d}t \,c(t):= {\rm d}t \,a(t) \frac{r-
\nu \bar m}{\mu}(1-e^{-\mu(T-t)}).
\end{equation}}
% The expected total capital generated by each additional client by time $T$ (in the conditioned process) is approximately\footnote{MM: how is $c(t)$ defined? Or (and this is what I actually guess) is this perhaps somehow the {\it definition} of $c(t)$?}
% \[
% {\rm d}t \,c(t):={\rm d}t \,a(t) \int^{T-t}_0 e^{-\mu x}(r - \nu \bar m)\,{\rm d}t = {\rm d}t \,a(t) \frac{r-
% \nu \bar m}{\mu}(1-e^{-\mu(T-t)}), 
% \]
% where this approximation holds for $\varepsilon$ and ${\rm d}t$ small.
The share of the total rate $I_{[0,T]}(f^{(\star,T)},g^{(\star,T)})$ that can be attributed to these additional clients is
%\footnote{MM: here I start to lose track of what is going on. In the middle display, what is $x$? And what is the justification of the last `$\approx$'? (The second `$\approx$' is `Taylor', right?)}
\begin{equation}\label{CCO2}
\begin{aligned}
I_{t, t+{\rm d}t} ( (f^{(\star, T)}(t), f^{(\star, T)}(t+{\rm d}t)),\mathbb{R}^2) &\approx I_{t, t+{\rm d}t} ( (\bar f(t), \bar f(t+{\rm d}t)+{\rm d}t\, a(t) ),\mathbb{R}^2) \\
&\approx ({\rm d}t)^2 a(t)^2  \left. \frac{\partial^2}{\partial y^2} I_{(t,t+{\rm d}t)}((\bar f(x), y),\mathbb{R}^2)\right|_{y=\bar f(x+{\rm d}t)} \\
&\approx \frac{{\rm d}t\, a(t)^2}{\lambda + \bar f(t)\mu},
\end{aligned}
\end{equation}
{where the first step requires $\varepsilon$ to be small, the second step follows from a Taylor expansion and requires ${\rm d}t$ to be small, and the final step follows from the fact that the second derivative of the Legendre transform evaluated at its mean is the reciprocal of the variance of the underlying random variable.}
% where these approximations are valid in the regime that $\varepsilon$ and ${\rm d}t$ are small.
In view of \eqref{CCO1} and \eqref{CCO2}, the marginal increase in rate per unit capital corresponding to increasing or decreasing the additional clients that arrive in $[t, t+{\rm d}t)$ is
\begin{align}
\begin{split}\label{eq:OR1}
    \frac{{\rm d}}{{\rm d}c(t)} \left( \frac{ a(t)^2}{\lambda + \bar f(t)\mu} \right) &= \frac{{\rm d}}{{\rm d}c(t)} \left( \frac{c(t)^2}{(\lambda + \bar f(t)\mu)\left[\frac{r-
\nu \bar m}{\mu}(1-e^{-\mu(T-t)}) \right]^2 } \right) \\
&= \frac{2c(t) }{(\lambda + \bar f(t)\mu)\left[\frac{r-
\nu \bar m}{\mu}(1-e^{-\mu(T-t)}) \right]^2 }.
\end{split}
\end{align}

% The total additional capital that can be attributed to the additional clients is approximated by
% \[
% {\rm d}t a(t) \int^{T-t}_0 e^{-\mu x}(r - \nu \bar m){\rm d}t = {\rm d}t a(t) \frac{r-
% \nu \bar m}{\mu}(1-e^{-\mu(T-t)}), 
% \]
% where this approximation again holds for $\varepsilon$ and ${\rm d}t$ small.
% The additional capital per unit rate associated to the additional clients of $[t, t+{\rm d}t)$ is then
% \begin{equation}\label{eq:OR1}
% \frac{(r-\nu \bar m)(1-e^{-\mu(T-t)})(\lambda + \bar f(t)\mu)}{a(t)\mu}.
% \end{equation}

Now observe that the additional capital that can be attributed to clients generating fewer total claims than expected is 
\[
b := g^{(\star, T)}(T) - (r - \nu \bar m) \int^T_0 f^{(\star, T)}(t) {\rm d}t,
\]
where we recall that $g^{(\star, T)}(T)=\bar g(T)+\varepsilon$.
The rate associated with these reduced total claims is 
\begin{align}
    I_{[0,T]}(g^{(\star,T)}|f^{(\star,T)}) &= \bar K_{f^{(\star,T)}}(g^{(\star, T)}(T)) \label{eq:tc1} \\
    &\approx \bar K_{\bar f}(\bar g(T)+b),\label{eq:tc2}
\end{align}
where 
\begin{equation}\label{eq:tc3}
    \bar K_{f}(x) = \sup_\theta (\theta x - \gamma(\theta)), \quad \text{with} \quad     \gamma(\theta)= \exp\left(\nu \int^T_0 f(t) {\rm d}t \,(\beta(\theta) -1)\right).
\end{equation}
% with 
% \begin{equation}\label{eq:tc4}
%     \gamma(\theta)= \exp\left(\nu \int^T_0 f(t) {\rm d}t (\beta(\theta) -1)\right).
% \end{equation}
Note that \eqref{eq:tc1} follows by the contraction principle, while {\eqref{eq:tc2} uses $f^{(\star,T)}(\cdot)\approx \bar f(\cdot)$ for $\varepsilon$ small}. In addition, \eqref{eq:tc3} follows from that fact that given the client population $f$ the total number of claims is Poisson with {mean} $\nu \int_{0}^Tf(t)\,{\rm d}t$. For $\varepsilon$ small we then have 
\begin{equation}\label{CCO3}
\bar K_{\bar f}(\bar g(T)+b) \approx \frac{b^2}{ \beta''(0)  \nu \int^{T}_0 \bar f(t) \,{\rm d}t},
\end{equation}
{where we again apply a Taylor expansion, and use the fact that the second derivative of a Legendre transform is the reciprocal of the variance of the underlying random variable.}
{In view of \eqref{CCO3}, the} marginal increase in rate per unit capital corresponding to increasing or decreasing $b$ is
\begin{equation}\label{eq:OR2}
\frac{{\rm d}}{{\rm d}b} \left( \frac{b^2}{ \beta''(0)  \nu \int^{T}_0 \bar f(t) \,{\rm d}t} \right) = \frac{2b}{ \beta''(0)  \nu \int^{T}_0 \bar f(t)\, {\rm d}t}
\end{equation}

By the optimality of $(f^{(\star,T)},g^{(\star,T)})$, and Equations \eqref{eq:OR1} and \eqref{eq:OR2} we have 
\[
\frac{2 c(t) }{(\lambda + \bar f(t)\mu)\left[\frac{r-
\nu \bar m}{\mu}(1-e^{-\mu(T-t)}) \right]^2 } = \frac{2b}{ \beta''(0)  \nu \int^{T}_0 \bar f(t)\, {\rm d}t}, \quad \text{for all } t \in [0,T],
\]
so that 
\[
c(t) = \frac{b(\lambda + \bar f(t)\mu)\left[\frac{r-
\nu \bar m}{\mu}(1-e^{-\mu(T-t)}) \right]^2}{\beta''(0)  \nu \int^{T}_0 \bar f(t)\, {\rm d}t}, \quad \text{for all } t \in [0,T].
\]
Because the total additional capital must be $\varepsilon$, we have 
\begin{equation*}
    \varepsilon = b+ \int_0^T c(t)\, {\rm d}t ,
\end{equation*}
so that 
\[
\lim_{\varepsilon \downarrow 0}E_1(a,T) = \frac{\int^T_0 c(t) \,{\rm d}t}{b+ \int_0^T c(t)\, {\rm d}t} = \frac{\int^T_0(\lambda + \bar f(t)\mu)\left[\frac{r-
\nu \bar m}{\mu}(1-e^{-\mu(T-t)}) \right]^2 {\rm d}t }{ \beta''(0)  \nu \int^{T}_0 \bar f(t) \,{\rm d}t+\int^T_0(\lambda + \bar f(t)\mu)\left[\frac{r-
\nu \bar m}{\mu}(1-e^{-\mu(T-t)}) \right]^2 {\rm d}t } ,
\]
where we observe that right-hand side equals \eqref{ECT}, as desired.

% Below we will argue under the assumption that $\varepsilon$ is small.

% The additional clients that can be attributed to the time interval $[t, t+{\rm d}t)$ are 
% \[
% (f^{(\star,T)'}(t) - \bar f'(t) + O(\varepsilon)) {\rm d}t.
% \]
% The rate that can be attributed to these additional clients is 
% \begin{align*}
% I_{(t,t+{\rm d}t)} &(\bar f(x), \bar f(x+{\rm d}t)+{\rm d}t( f^{(\star,T)'}(t) - \bar f'(t) +O(\varepsilon)))\\
% &=  [{\rm d}t( f^{(\star,T)'}(t) - \bar f'(t) +O(\varepsilon))]^2 \left. \frac{\partial^2}{\partial y^2} I_{(t,t+{\rm d}t)}(\bar f(x), y)\right|_{y=\bar f(x+{\rm d}t)}+o(({\rm d}t)^2)\\
% &=\frac{{\rm d}t (f^{(\star,T)'}(t) - \bar f'(t))^2}{(\lambda + \bar f(t))\mu}+ {\rm d}t O(\varepsilon)+o(({\rm d}t)^2).
% \end{align*}
% The additional capital (per-client) that can be attributed to these additional clients is 
% \[
% \int^{T-t}_0 e^{-\mu x}(r - \nu \bar m){\rm d}t + O(\varepsilon)= \frac{r-
% \nu \bar m}{\mu}(1-e^{-\mu(T-t)})+O(\varepsilon). 
% \]

\end{document}